\documentclass[11pt,oneside]{article}

\usepackage{amsthm, amsmath, amssymb, amsfonts,epsfig}

\usepackage{epstopdf}

\usepackage{graphicx}
\usepackage[font=sl,labelfont=bf]{caption}
\usepackage{subcaption}

\usepackage{slashbox}

\usepackage{enumerate}

\usepackage[colorlinks=true, pdfstartview=FitV, linkcolor=blue,
            citecolor=blue, urlcolor=blue]{hyperref}
\usepackage[usenames]{color}
\definecolor{Red}{rgb}{0.7,0,0.1}
\definecolor{Green}{rgb}{0,0.7,0}

\usepackage[numbers]{natbib}

\usepackage{url}
\makeatletter\def\url@leostyle{%
 \@ifundefined{selectfont}{\def\UrlFont{\sf}}{\def\UrlFont{\scriptsize\ttfamily}}} \makeatother

\usepackage{accents, comment}

\usepackage{bbm}

\usepackage{mathrsfs}

\newtheorem{theorem}{Theorem}
\newtheorem{proposition}[theorem]{Proposition}
\newtheorem{lemma}[theorem]{Lemma}

\theoremstyle{definition}
\newtheorem{definition}[theorem]{Definition}

\theoremstyle{remark}
\newtheorem{remark}[theorem]{Remark}

\numberwithin{equation}{section}
\numberwithin{theorem}{section}

\def\cB{\mathcal{B}}

\def\cF{\mathcal{F}}

\def\cH{\mathcal{H}}

\def\cK{\mathcal{K}}
\def\cL{\mathcal{L}}

\def\cN{\mathcal{N}}

\def\cR{\mathcal{R}}

\def\cZ{\mathcal{Z}}

\def\bE{\mathbb{E}}

\def\bN{\mathbb{N}}

\def\bP{\mathbb{P}}
\def\bQ{\mathbb{Q}}
\def\bR{\mathbb{R}}

\newcommand{\1}{\mathbbm{1}}                     
\newcommand{\set}[1]{\{#1\}}            
\renewcommand{\d}{\operatorname{d}\!}   

\title{Hypothesis testing for stochastic PDEs driven by additive noise}
\author{
Igor Cialenco\thanks{Research of IC was partially supported by NSF grant DMS-1211256.} \\
\small  Department of Applied Mathematics \\[-0.6ex]
\small  Illinois Institute of Technology \\[-0.6ex]
\small  10 West 32nd Str, Bld E1, Room 208 \\[-0.6ex]
\small  Chicago, IL 60616-3793  \\[-0.6ex]
\small  \url{igor@math.iit.edu}
\and 
 Liaosha Xu\\
\small  Department of Applied Mathematics \\[-0.6ex]
\small  Illinois Institute of Technology \\[-0.6ex]
\small  10 West 32nd Str, Bld E1, Room 208 \\[-0.6ex]
\small  Chicago, IL 60616-3793  \\[-0.6ex]
\small  \url{lxu29@hawk.iit.edu}
}

\date{First Circulated: August 08, 2013\\ This version: May 22, 2014}
\begin{document}
\maketitle

\begin{abstract}
\noindent
We study the simple hypothesis testing problem for the drift coefficient for stochastic fractional heat equation driven by additive noise. We introduce the notion of asymptotically the most powerful test, and find explicit forms of such tests in two asymptotic regimes: large time asymptotics, and increasing number of Fourier modes. The proposed statistics are derived based on Maximum Likelihood Ratio. Additionally, we obtain a series of important technical results of independent interest: we find the cumulant generating function of the log-likelihood ratio; obtain sharp large deviation type results for $T\to\infty$ and $N\to\infty$.

\bigskip
{\noindent \small

}


\end{abstract}

\newpage

\tableofcontents

\newpage

\section{Introduction}
In this paper we study a simple hypothesis testing problem for drift (viscosity) coefficient for some linear parabolic PDEs driven by an additive noise (white in time and possible colored in space).
The underlying assumption is that one path of the solution is observed continuously in time as an element of an infinite dimensional space.
The problem of estimating the drift coefficient in this setup, and assuming that all other parameters are known, has been studied by several authors, starting with the seminal paper \citet{HubnerRozovskiiKhasminskii}.
It is known that the Maximum Likelihood Estimator (MLE) constructed from the projection of the solution on the space spanned by the first $N$ Fourier modes observed over time interval $[0,T]$, is consistent and asymptotically normal in both regimes: as number of Fourier coefficients of the solution tends to infinity and time horizon $T$ is fixed and finite (cf. the survey paper \citet{Lototsky2009Survey} and references therein), or as $T\to\infty$ and $N$ is fixed and finite (the proof is presented in this work).
While over the last decade significant progress has been made in studying asymptotic properties of the MLE type estimators for various classes of SPDEs,
and many properties have been well understood, the problem of hypothesis testing and goodness-of-fit tests for SPDEs remains an open field.
This paper is the first attempt to address this clearly important and, as it turns out, also challenging problem.
We take the similar  spectral approach as in the papers mentioned above, and consider the simple hypothesis testing problem $\theta=\theta_0$ versus $\theta=\theta_1$ for the drift coefficient $\theta$.
Having an MLE and its consistency and asymptotic normality at hand, as one may expect, likelihood ratio test and Neyman-Pearson type lemma should answer the question.
Indeed, such result can be easily derived, however, due to the fact that the distribution of the likelihood ratio is hard to find, the common problem is to find the threshold of the likelihood test, for which apparently there is no an explicit solution.
There are several methods to overcome this problem, one of which is the asymptotic approach with the class of tests of given asymptotic level - the method that we take as a baseline too.
For the case of finite dimensional ergodic diffusion processes with continues time observation, the similar approach for large time asymptotics can be traced back to \citet{Kutoyants1975}; for more discussions see also \citet{KutoyantsBook2004} and references therein\footnote{We should mention that there is a significant literature devoted to goodness-of-fit tests for diffusion processes, which we will not list here.}. We consider the asymptotics in two regimes mentioned above: (i) large time asymptotics, $T\to\infty$ with $N$ fixed;  (ii)  number of Fourier modes increases, $N\to\infty$, with fixed $T$.
As a novel step into this direction, we introduce the notion of asymptotically the most powerful test in a given class.
The main contribution of the paper is identification of the `right class' of tests with a given (asymptotic) significance level, and then we find (asymptotically) the most powerful test in this class that can be computed explicitly.
The main challenge is to control the power of the tests as $T\to\infty$ or $N\to\infty$. For the case of large time asymptotics, the ideas are rooted to some results and methods from large deviations for ergodic stochastic processes. We develop some new results on sharp large deviations tailored to our needs. Since the projected system of the first $N$ Fourier modes essentially represents an $N$ dimensional system of stochastic differential equations, the obtained results for the case of large time asymptotics can be applied in particular to finite dimensional ergodic processes too.
For the case $N\to\infty$, while the approach is similar to that from large time asymptotics, we establish some new results that informally can be called (sharp) large deviations principles in number of Fourier modes $N$ of the underlying observed solution of the considered SPDE. Using this, we derive the form of asymptotically the most powerful test in a suitable class of tests with a given asymptotic level. We also believe that the obtained results will be instrumental and serve as a foundation for future developments in the area of hypothesis testing and goodness-of-fit tests for SPDEs.

It is true that the main results are based on continuous time sampling, and may appear as being mostly of theoretical interest.
Eventually, in real life experiments, the random field would be measured/sampled on a discrete grid, both in time and spatial domain.
However, the main ideas of this paper have a good prospect to be applied to the case of discrete sampling too.
If we assume that the first $N$ Fourier modes are observed at some discrete time points, then, to apply the theory presented here, one essentially has to approximate some integrals, including some stochastic integrals, convergence of each is well understood.
Of course, the exact rates of convergence still need to be established.
In real life applications, given the current technological possibilities, usually data is sampled in the time component at sufficiently high frequencies (climate, oceanography, finance etc).
The challenge comes with spacial data, which can be sparse and not evenly distributed.
The connection between discrete observation in space and the approximation of Fourier coefficients is more intricate. Natural way is to use discrete Fourier transform for such approximations. While intuitively clear that increasing the number of observed spacial points will yield to the computation of larger number of Fourier coefficients (number of points from physical domain approximatively corresponds to the number of Fourier modes), it is less obvious, in our opinion, how to prove consistency of the estimators, asymptotic normality, and corresponding properties from hypothesis testing problem.
Thus, both asymptotic regimes considered here, and in the existing literature, are relevant for practical proposes.
In a follow up paper \cite{CialencoXu2013-2}, we present some numerical results on estimation of statistical errors, including for the tests proposed in this paper.
Moreover, using the sharp large deviations results developed here, as well as the underlying ideas of establishing the `optimal' class of tests, we derive a new likelihood ratio type test, that may not be the most powerful, but which allows to have exact control on the Type~I and Type~II errors for finite $T$ and $N$.

The paper is organized as follows. In Section~\ref{sec:MathSetting} we introduce the main object of our study, fractional\footnote{By fractional we mean that the negative of Laplace operator is taken with fractional powers.} stochastic heat equation driven by additive space-time white noise. We provide sufficient conditions for the existence and uniqueness of the solution.  In Section~\ref{sec:ConsistAndNormality} we set up the statistical problem of estimation of the drift coefficient, derive the MLE estimators, and establish the consistency and asymptotic normality of these estimators. For sake of completeness we prove the asymptotic properties of MLEs for large time asymptotics case; a result considered as known but to the best of our knowledge not shown rigourously in the existing literature.
Section~\ref{sec:HypTest} we formulate the problem and present the main results of our study.
We start with a reasonable class of rejection regions based on the likelihood-ratio statistics, and prove a version of Neyman-Pearson Lemma for such statistics.
Also here we discuss the deficiencies of this class of tests. In Section~\ref{sec:mainResults} we introduce the relevant classes of tests, and present the main results (without proofs).
The detailed proofs are presented in Section~\ref{sec:proofs}. Each asymptotic regime is considered in a separate subsection - Section~\ref{ssec:AsymptoticT} for the case of large time asymptotics, and respectively Section~\ref{sec:AsympMethodN} for large number of Fourier modes.
We took the road of rather heuristic and linear exposition of the method, by starting with a natural candidate, denoted by  $\cK_\alpha^*$, as a class of tests. 
Using some existing results from theory of large deviations for ergodic processes, we hint why this class is not good for our purposes.
Section~\ref{ssec:CumFuncT} contains some technical results, and, in particular, we find the cumulant generating function of the log-likelihood ratio, after appropriate use of Feynman-Kac formula and solving  the corresponding PDE; a result itself of independent interest. 
Also here, we obtain a sharp large deviation result for the log-likelihood ratio.  Consequently, we show that there are statistics in $\cK_\alpha^*$ with higher power than the naturally derived likelihood ratio ones.
Moreover, we show why the natural choice of class of tests $\cK_\alpha^*$ in principle  is not a reasonable one.
Finally, in Section~\ref{sec:newClassT}, we present the new class of rejection regions and prove the main result for large times. 
Taking advantage of the heuristic exposition from previous section, and since the general road of deriving the reasonable class of test for $N\to\infty$ is similar to the case $T\to\infty$, we start with a series of technical lemmas, then prove the main results.
We want to mention, that although the general agenda for $N$-case is similar to $T$-case, most of the auxiliary results were proved by different methods.
For convenience, in Appendix~\ref{sec:appendix} we present some known results, as well as the proofs of selected technical results from this paper.

As mentioned above, this is the first attempt to address the hypothesis testing problem for SPDEs, and clearly many problems are still open.
First natural problem is to consider the case when both $T,N$ are getting large simultaneously. From practical point of view this can be less important, however from theoretical point of view this is an important problem, with potentially nontrivial technical challenges.
In the present work, we considered rejection regions of the simplest form - likelihood ratio larger than a threshold. Generally speaking, one can consider rejection regions of the form `likelihood ratio belongs to a Borel set'. This generalization, in particular, can produce tests with `faster convergence rate', as some of our preliminary results show. Clearly, a different method of finding the threshold $c_\alpha$ from Neyman-Pearson Lemma is to employ some numerical methods, such as (quasi) Monte Carlo.
Another natural follow up problem is to consider composite hypothesis. Using the results from this paper, one can derive similar results for hypothesis testing problem of the form $\theta=\theta_0$ vs $\theta>\theta_1$ with $\theta_0\neq \theta_1$. The authors plan to address some of these problems in their future works.

\section{Mathematical Setting}\label{sec:MathSetting}
Let $(\Omega,\cF,\set{\cF_t}_{t\geq 0}, \bP)$ be a stochastic basis with usual assumptions, and let $\set{w_j,\ j\geq 1}$ be a collection of independent standard Brownian motions on this basis.
Let $G$ be a bounded and smooth domain in $\bR^d$, and let us denote by $\Delta$ the Laplace operator on $G$ with zero boundary conditions.
The corresponding scale of Sobolev spaces will be denoted by $H^s(G)$, or simply $H^s$, for $s\in\bR$.
It is well known (cf. \citet{Shubin}) that:
a) the set $\set{h_k}_{k\in\bN}$ of eigenfunctions of $\Delta$ forms a complete orthonormal system in $L^2(G)$;
b) the corresponding eigenvalues $\rho_k,k\in\bN$, can be arranged such that $0<-\rho_1\leq - \rho_2\leq \ldots$, and there exists a positive constant $\varpi$ so that
\begin{equation}\label{eq:asymEigenv}
  \lim_{k\to\infty} |\rho_k|k^{-2/d} = \varpi.
\end{equation}
In what follows we will use the notation $\lambda_k:=\sqrt{-\rho_k}, \ k\in\bN$, and $\Lambda=\sqrt{-\Delta}$.
Also, for two sequences of numbers $\{a_n\}$ and $\{b_n\}$, we will write $a_n \sim b_n$, if there exists a nonzero and finite number $c$ such that $\lim_{n\to\infty}a_n/b_n=c$.

We consider the following Stochastic PDE
\begin{equation}\label{eq:mainSPDE}
\d U(t,x) + \theta (-\Delta)^\beta U(t,x)\d t = \sigma \sum_{k\in\bN} \lambda_k^{-\gamma}h_k(x)\d w_k(t), \quad t\in[0,T], \ U(0,x) = U_0, \ x\in G,
\end{equation}
where $\theta>0$, $\beta>0, \ \gamma \geq 0$, $\sigma\in\bR\setminus\set{0}$, and $U_0\in H^s(G)$ for some $s\in\bR$.

Using standard arguments (cf. \citet{RozovskiiBook,ChowBook}), it can be proved that if $2(\gamma-s)/d > 1$, then \eqref{eq:mainSPDE} has a unique  solution (weak in PDE sense, and strong in probability sense)
\begin{equation}\label{eq:existuniqueSol}
U\in L_2(\Omega\times[0,T]; H^{s+\beta})\cap L^{2}(\Omega; C((0,T); H^s)).
\end{equation}

\begin{remark}

We would like to mention that the obtained results can be stated in a more general setup similar to that from \citet{HubnerRozovskiiKhasminskii,HuebnerLototskyRozovskii97,Lototsky2009Survey}.
The authors decided to take powers of negative Laplace operator instead of a general positive defined linear operator that generates a scale of Hilbert spaces (cf. the setup from \citet{CialencoLototsky2009}) with eigenfunctions forming a complete orthonormal system. The only essential difference in our analysis would be that the results and their derivation would be written in the terms of the asymptotics of the eigenvalues of that positive defined operator. Since for Laplace operator the asymptotic behavior of eigenvalues  is rather simple, and given by \eqref{eq:asymEigenv}, the exposition of the results becomes easier to follow, in our opinion.
\end{remark}

\subsection{Consistency and Asymptotic Normality}\label{sec:ConsistAndNormality}
Throughout, we will assume that $s\geq0,$ and thus $U\in L_2(\Omega\times[0,T]; H^{\beta})\cap L^{2}(\Omega; C((0,T); H^0))$.
We denote by $u_k,k\in\bN,$ the Fourier coefficient of the solution $u$ of \eqref{eq:mainSPDE} with respect to $h_k,k\in\bN$, i.e. $u_k(t) = (U(t),h_k)_0, k\in\bN$.
Let $H_N$ be the finite dimensional subspace of $L_2$ generated by $\set{h_k}_{k=1}^N$, and denote by $P_N$ the projection operator of $L_2$ into $H_N$, and put $U^N = P_NU$, or equivalently $U^N:=(u_1,\ldots,u_N)$. Note that each Fourier mode $u_k,k\in\bN$, is an Ornstein-Uhlenbeck process with dynamics given by
\begin{equation}\label{eq:OU-Fourier}
\d u_k = -\theta \lambda_k^{2\beta} u_k\d t + \sigma \lambda_k^{-\gamma} \d w_k(t), \quad  u_k(0) = (U_0,h_k), \ t\geq 0.
\end{equation}

We denote by $\bP^{N,T}_{\theta}$ the probability measure on $C([0,T]; H_N)\backsimeq C([0,T]; \bR^N)$ generated by the $U^N$.
The measures $\bP^{N,T}_{\theta}$ are equivalent for different values of the parameter $\theta$, and the Radon-Nikodym derivative, or Likelihood Ratio, has the form
\begin{align}
L(\theta_0,\theta;U^N_T)=\frac{\bP^{N,T}_{\theta}}{\bP^{N,T}_{\theta_{0}}} =
\exp\Big(-(\theta-\theta_0)\sigma^{-2} & \sum_{k=1}^N\lambda_k^{2\beta+2\gamma}\big(\int_0^T  u_k(t)du_k(t) \nonumber \\
 & +\frac{1}{2}(\theta+\theta_0)\lambda_k^{2\beta}\int_0^Tu_k^2(t)dt\big)\Big). \label{eq:RadonNikodymUn}
\end{align}
Maximizing the Log of the likelihood ratio with respect to the parameter of interest $\theta$, we get the following Maximum Likelihood Estimator (MLE)
\begin{equation}\label{eq:MLE-UN}
\widehat{\theta}_T^N = -\frac{\sum_{k=1}^{N}\lambda_k^{2\beta+2\gamma}\int_0^T u_k(t)du_k(t)}{\sum_{k=1}^{N}\lambda_k^{4\beta+2\gamma}\int_0^T u_k^2(t)dt},
\quad N\in\bN, T>0.
\end{equation}

\begin{theorem}
Assume that $2\gamma>d$. Then,
\begin{enumerate}
  \item For every $N\in\bN$ and $T>0$, the estimator $\widehat{\theta}^N_T$ is an unbiased estimator of $\theta$.
  \item  (Consistency and asymptotic normality in number of Fourier coefficients) \\
  For every fixed $T>0$,
  \begin{equation}\label{eq:ConstN}
  \lim_{N\to\infty} \widehat{\theta}^N_T = \theta, \quad \textrm{a.e.}
  \end{equation}
  and
  \begin{equation}\label{eq:AsymNormN}
 \lim_{N\to\infty} N^{\beta/d+\frac{1}{2}} \left( \widehat{\theta}^N_T -\theta\right) \overset{d}= \cN\left(0,\frac{(4\beta/d+2)\theta}{\varpi^{\beta}\sigma^2T}\right).
  \end{equation}
  \item (Consistency and asymptotic normality in large time asymptotics) \\
  For every fixed $N\in\bN$,
  \begin{equation}\label{eq:ConstT}
  \lim_{T\to\infty} \widehat{\theta}^N_T = \theta, \quad \textrm{a.e.}
  \end{equation}
  and
  \begin{equation}\label{eq:AsymNormT}
 \lim_{T\to\infty} \sqrt{T} \left( \widehat{\theta}^N_T -\theta\right) \overset{d}= \cN(0,2\theta/M),
  \end{equation}
  where $M=\sum_{k=1}^N\lambda_k^{2\beta}$.
\end{enumerate}
\end{theorem}
\begin{proof}
Using the dynamics of $u_k$ given in \eqref{eq:OU-Fourier}, the estimator $\theta_T^N$ can be conveniently  represented as follows
\begin{equation}\label{eq:theta-conv}
\widehat{\theta}_T^N = \theta - \frac{\sigma\sum_{k=1}^N \lambda_k^{2\beta+\gamma}\int_0^Tu_k\d w_k}{ \sum_{k=1}^N \lambda_k^{4\beta+2\gamma} \int_0^Tu_k^2\d t}.
\end{equation}
From this, we conclude immediately that $\widehat \theta_T^N$ is an unbiased estimator of $\theta$.
For the second part of the theorem, we refer the reader to the original paper \citet{HuebnerRozovskii} or the survey paper \citet{Lototsky2009Survey}.

For sake of completeness we present here the proof for large time asymptotics, Part 3 of the theorem.
We start with consistency.
Taking into account the representation \eqref{eq:theta-conv}, it is enough to prove that, for any $k\in\bN$,
\begin{equation}\label{eq:duOverdW}
\lim_{T\to\infty}\frac{\int_0^Tu_kdw_k}{\int_0^Tu_k^2dt}=0 \quad\textrm{a.s.}.
\end{equation}

Following the idea from \citet[Chapter 17]{LiptserShiryayevBook1978}, we consider the following stopping times $\tau_s:=\inf\{t: \int_0^tu_k^2(r)dr>s\}$,  for $s\geq 0$.
We claim that the stochastic process
$$
v_k(s)=\int_0^{\tau_s}u_kdw_k,
$$
is a Wiener process with respect to the filtration $\set{\mathcal{G}_s:=\mathcal{F}_{\tau_s}, \ s\geq 0}$.
Indeed, by It\^o's formula, we have
$$
de^{i\lambda x(t)}=i\lambda e^{i\lambda x(t)}u_kdw_k-\frac{\lambda^2}{2}e^{i\lambda x(t)}u_k^2dt,
$$
where $x(t)=\int_0^tu_kdw_k$. Integrating from $\tau_{s_1}$ to $\tau_{s_2}$, we obtain
\begin{align}\label{eq:eetaus}
e^{i\lambda v_k(s_2)}=e^{i\lambda v_k(s_1)}+i\lambda \int_{\tau_{s_1}}^{\tau_{s_2}}e^{i\lambda x(t)}u_kdw_k-\frac{\lambda^2}{2}\int_{\tau_{s_1}}^{\tau_{s_2}}e^{i\lambda x(t)}u_k^2dt.
\end{align}
By making the substitution $t=\tau_s$ in the above Lebesgue integral, we get
\begin{align*}
\int_{\tau_{s_1}}^{\tau_{s_2}}e^{i\lambda x(t)}u_k^2dt=\int_{s_1}^{s_2}e^{i\lambda x(\tau_s)}u_k^2(\tau_s)d\tau_s=\int_{s_1}^{s_2}e^{i\lambda v_k(s)}ds.
\end{align*}
Using this, and the fact that
$$
\bE\left[\int_{\tau_{s_1}}^{\tau_{s_2}}e^{i\lambda x(t)}u_kdw_k\bigg|\mathcal{G}_{s_1}\right]=0,
$$
equation \eqref{eq:eetaus} yields the following integral equation
$$
V(s_2)=1-\frac{\lambda^2}{2}\int_{s_1}^{s_2}V(s)ds,
$$
where $V(s):=\bE\left[e^{i\lambda (v_k(s)-v_k(s_1))}\big|\mathcal{G}_{s_1}\right]$.
Solving this integral equation we obtain that
$$
V(s_2)=\bE\left[e^{i\lambda (v_k(s_2)-v_k(s_1))}\big|\mathcal{G}_{s_1}\right]=e^{-\lambda^2 (s_2-s_1)/2},
$$
which means $v_k(t)$ is a Gaussian martingale and $\bE\left[(v_k(s_2)-v_k(s_1))^2\big|\mathcal{G}_{s_1}\right]=s_2-s_1$.
Consequently, from normality, it follows that  $\bE(v_k(s_2)-v_k(s_1))^4=3(s_2-s_1)^2$, and hence by Kolmogorov Criterion the process admits a continuous version.
Therefore,  $v_k(t)$ is a Wiener process.

By the definition of $\tau_s$, we notice that $\tau_s$ is non-decreasing with respect to $s$ a.s., and thus  $\lim_{s\to\infty}\tau_s$ exists (possible equal to infinity) with probability one.
On the other hand, since $u_k(t)$ is a continuous function on $\bR_+$, we get that $\int_0^{T}u_k^2dt<+\infty$ a.s. for any finite $T\in(0,\infty)$. This implies that $\tau_s$ cannot be bounded as  $s\to\infty$. Therefore, we conclude that  $\lim_{s\to\infty}\tau_s=+\infty$ a.s.

Finally, by the law of iterated logarithm,
$$\lim_{T\to\infty}\frac{\int_0^Tu_kdw_k}{\int_0^Tu_k^2dt}=\underset{s\to\infty}{\lim}\frac{\int_0^{\tau_s}u_kdw_k}{\int_0^{\tau_s}u_k^2dt}= \underset{s\to\infty}{\lim}\frac{v_k(s)}{s}=0 \quad\textrm{a.s}.
$$
Thus, \eqref{eq:ConstT} is established.

By direct evaluations, one can check that $u_k$ satisfies Conditions~\eqref{eq:CondRP}, and thus, by Theorem~\ref{th:LLN} we have that $u_k$ is ergodic with invariant density
$$
f_k(x)=\lambda_k^{\beta+\gamma}\sqrt{\frac{\theta}{\pi\sigma^2}}\exp\left(-\theta\sigma^{-2}\lambda_k^{2\beta+2\gamma} x^2\right),
$$
and hence
\begin{align}\label{eq:5}
\lim\limits_{T\to\infty}\frac{1}{T}\int_0^T\lambda_k^{4\beta+2\gamma}u_k^2dt =\frac{\sigma^2}{2\theta}\lambda_k^{2\beta}\quad \textrm{a.s}.
\end{align}
Summing up \eqref{eq:5} from 1 to $N$,  we obtain
\begin{align}\label{eq:6}
\lim_{T\to\infty}\frac{1}{T}\sum_{k=1}^N\lambda_k^{4\beta+2\gamma}\int_0^Tu_k^2dt = \sigma^2M/(2\theta)\quad\textrm{a.s.}
\end{align}
By Theorem~\ref{th:CLT} we also have
\begin{align}\label{eq:3}
X_k^T:=\frac{1}{\sqrt{T}}\int_0^T\lambda_k^{2\beta+\gamma}u_kdw_k \overset{d}\longrightarrow \mathcal{N}\left(0,\frac{\sigma^2}{2\theta}\lambda_k^{2\beta}\right).
\end{align}
Let $Y_T:=\sum_{k=1}^NX_k^T$, and let $\varphi_{X_k}^T(v)$ and $\varphi_{Y}^T(v)$ denote the characteristic functions of $X_k^T$ and $Y_T$ respectively.
Since  $X_k^T, \ k\in\bN$, are independent, we have that
\begin{align}\label{eq:4}
\varphi_{Y}^T=\prod_{k=1}^N\varphi_{X_k}^T.
\end{align}
Consequently, by \eqref{eq:3} and L\'evy's Continuity Theorem (cf. \citet[Chapter 19]{JacodProtterBook2003}), we deduce that
$$
\lim_{T\to\infty}\varphi_{X_k}^T=\exp\left(-\frac{v^2\sigma^2}{4\theta}\lambda_k^{2\beta}\right),
$$
which combined with $(\ref{eq:4})$ gives that
$$
\lim_{T\to\infty}\varphi_{Y}^T=\exp\left(-\frac{v^2\sigma^2}{4\theta}M\right),
$$
Using L\'evy's Continuity Theorem one more time,  we have
\begin{align}\label{eq:7}
Y_T=\frac{1}{\sqrt{T}}\sum_{k=1}^N\lambda_k^{2\beta+\gamma}\int_0^Tu_kdw_k \overset{d}\longrightarrow \mathcal{N}(0,\sigma^2M/(2\theta)).
\end{align}
Notice that $\sqrt{T}\left(\widehat{\theta}_T^N-\theta\right)=Y_T/Z_T$, where $Z_T:=\frac{1}{T}\sum_{k=1}^N\lambda_k^{4\beta+2\gamma}\int_0^Tu_k^2dt$.
Finally, by \eqref{eq:7}, \eqref{eq:6} and Slutsky's Theorem, the asymptotic normality \eqref{eq:AsymNormT} follows, and this concludes the proof.
\end{proof}

\section{Hypothesis Testing}\label{sec:HypTest}
In this section we will formulate the problem, and present the main findings of our study. 

\subsection{Formulation of the problem}\label{sec:formulation}
We will consider the problem of hypothesis testing for the drift coefficient $\theta$. In this work we focus our study on the case of a simple hypothesis in the continuous time observation framework: we will assume that the parameter $\theta$ can take only two values $\theta_0,\theta_1$, and that the observable is $U_T^N$ - the trajectory on $[0,T]$ of the projection of the solution on the first $N$ Fourier modes. Hence, we take the null and the alternative hypothesis as follows
\begin{align*}
\mathscr{H}_0 & :\quad \theta=\theta_0, \\
\mathscr{H}_1 &:\quad \theta=\theta_1.
\end{align*}
Without loss of generality, towards this end,  we will assume that $\theta_1>\theta_0$, and $\sigma>0$. Also, for simplicity, we will assume that $U_0=0$\footnote{Generally speaking, after appropriate modifications,  all results remain true if $U(0)\in L^2(G)$. However, for the deductions in Section~\ref{sec:AsympMethodN} we need additional assumption $U(0)\in H^{\beta+\gamma}(G)$.}.
Throughout, we fix a significance level $\alpha\in(0,1)$.
Suppose that $R\in\cB(C([0,T];\bR^N))$ is a rejection region for the test, i.e. if  $U_T^N\in R$ we reject the null and accept the alternative.
Naturally, we seek rejection regions with Type~I error $\bP^{N,T}_{\theta_0}(R)$ smaller than the significance level $\alpha$, and thus we consider the following class of rejection regions
$$
\mathcal{K}_{\alpha}:=\left\{R\in\cB(C([0,T];\bR^N)): \bP^{N,T}_{\theta_0}(R)\leq\alpha\right\}.
$$
The probability $\bP^{N,T}_{\theta_1}(R)$ of the true decision under $\mathscr{H}_1$  is called the \textit{power of the test}, and the goal is to find the \textit{most powerful rejection region} $R^*\in\cK_\alpha$ for the observation $U_T^N$.
Mathematically reciting, we give the following definition.
\begin{definition}\label{def:Rstar}
We say that a rejection region $R^*\in\cK_\alpha$ is the most powerful in the class $\mathcal{K}_{\alpha}$ if
$$
\bP^{N,T}_{\theta_1}(R)\le\bP^{N,T}_{\theta_1}(R^*),\qquad \textrm{ for all } R\in\mathcal{K}_{\alpha}.
$$
\end{definition}
As one may expect, once we have an MLE for the parameter of interest $\theta$, as well as its consistency, a Neyman-Pearson type lemma should give, at least a theoretical, answer to the hypothesis testing problem.

\begin{theorem}[Neyman-Pearson]\label{th:NPLemma}
Let $c_\alpha$ be a real number such that
\begin{align}\label{eq:c-alpha}
\bP^{N,T}_{\theta_0}(L(\theta_0,\theta_1,U_T^N)\ge c_{\alpha})=\alpha.
\end{align}
Then,
\begin{align}
R^*:=\{U_T^N: L(\theta_0,\theta_1,U_T^N)\ge c_{\alpha}\},
\end{align}
is the most powerful rejection region in the class $\mathcal{K}_{\alpha}$.
\end{theorem}
\begin{proof}
First note that such $c_\alpha$ exists, due to continuity of the distribution of the Log-Likelihood Ratio.
Let us assume that $R$ is  a rejection region with $\bP^{N,T}_{\theta_0}(R)\le\alpha$.
If $U_T^N\in R^*$, then $\1_{R^*}-\1_{R}\ge0$, and $L(\theta_0,\theta_1,U_T^N)-c_{\alpha}\ge0$.
Otherwise, if $U_T^N\notin R^*$, then $\1_{R^*} - \1_{R}\le0$, and $L(\theta_0,\theta_1,U_T^N)-c_{\alpha}<0$.
Thus, $\left(\1_{R^*}-\1_{R}\right)\left(L(\theta_0,\theta_1,U_T^N)-c_{\alpha}\right)\geq0$, and therefore, we have that
$$
\bE_{\theta_0}\left[\left(\1_{R^*}-\1_{R}\right)\left(L(\theta_0,\theta_1,U_T^N)-c_{\alpha}\right)\right]\ge0.
$$
This, combined with \eqref{eq:RadonNikodymUn} and \eqref{eq:c-alpha}, implies the following
\begin{align*}
\bP^{N,T}_{\theta_1}(R^*)-\bP^{N,T}_{\theta_1}(R)=&\bE_{\theta_1}\left(\1_{R^*}-\1_{R}\right)\\
=&\bE_{\theta_0}L(\theta_0,\theta_1,U_T^N)\left(\1_{R^*}-\1_{R}\right)\\
\ge&c_{\alpha}\bE_{\theta_0}\left(\1_{R^*}-\1_{R}\right)\\
=&c_{\alpha}\left(\bP^{N,T}_{\theta_0}(R^*)-\bP^{N,T}_{\theta_0}(R)\right)\\
=&c_{\alpha}\left(\alpha-\bP^{N,T}_{\theta_0}(R)\right)\ge0.
\end{align*}
This finishes the proof.
\end{proof}
Theorem~\ref{th:NPLemma} gives a complete theoretical answer to the hypothesis testing problem, however, generally speaking it is not possible to give an explicit formula for the constant $c_\alpha$. Given that MLE estimator $\widehat{\theta}^N_T$ is consistent and asymptotically normal, for both, large time asymptotics $T\to\infty$, and large space sample size $N\to\infty$, it is reasonable to take a large-sample or asymptotic test approach.
In what follows, we will study each case separately, starting with large time asymptotics, while fixing the number of Fourier modes $N$, and then in Section~\ref{sec:AsympMethodN} we will consider the case when the number of Fourier modes increases, while time horizon is fixed. Of course, eventually one can consider the case when both $T$ and $N$ converge to infinity, however, we will postpone this approach to further studies.

\subsection{Statement of main results}\label{sec:mainResults}
In this section we announce the main results along with their interpretations. The detailed proofs are deferred to subsequent sections.

The main goal is to establish the proper classes of tests, and consequently to find in these classes likelihood ratio type tests that are the most powerful tests in the asymptotic sense - the notion defined rigourously herein.
In what follows, we will mostly deal with families of rejection regions indexed either by time $T\in\bR_+$, while number of Fourier modes $N\in\bN$ is fixed, or by $N\in\bN$ while time horizon $T$ is fixed. If  no confusions arise, with slight abuse of notations, we will simply write $R_T$ or $(R_T)$ instead of $(R_T)_{T\in\bR_+}$, and since $N$ is fixed in this case, we omit it in our writings. Respectively, for the asymptotic regime $N\to\infty$, while $T$ is fixed, we write $R_N$ or $(R_N)$  instead of $(R_N)_{N\in\bN}$.

Next, we introduce the main concept of this paper:

\begin{definition}\label{def:mostpoweful}
For a fix $N\in\bN$, let $\cK$ be a generic set of rejection regions
$$
\mathcal{K} \subset \left\{(R_T)_{T\in\bR_+} \, : \, R_T\in\cB(C([0,T]; \bR^N)) \right\}.
$$
We say that the rejection region $(R_T^*)_{T\in\bR_+}\in\mathcal{K}$ is \textit{asymptotically the most powerful}, in the class $\mathcal{K}$, as $T\to\infty$,  if
\begin{align}\label{eq:asym-def}
\underset{T\to\infty}{\liminf} \frac{1-\bP^{N,T}_{\theta_1}(R_T)}{1-\bP^{N,T}_{\theta_1}(R_{T}^*)}\ge1,\qquad \textrm{ for all } (R_T)\in\mathcal{K}.
\end{align}

Analogously, for a fixed $T>0$, and with
$$
\widetilde{\cK} \subset \set{(R_N)_{N\in\bN} \, : \, R_N\in\cB(C([0,T]; \bR^N))},
$$
the rejection region $(\widetilde{R}_N)\in\widetilde{\mathcal{K}}$ is \textit{asymptotically the most powerful},  in the class $\widetilde{\mathcal{K}}$, as $N\to\infty$, if
\begin{align}\label{eq:asym-def-N}
\underset{N\to\infty}{\liminf} \frac{1-\bP^{N,T}_{\theta_1}(R_N)}{1-\bP^{N,T}_{\theta_1}(\widetilde{R}_N)}\ge1,\qquad \textrm{ for all } (R_N)\in \widetilde{\mathcal{K}}.
\end{align}
\end{definition}

The classes $\cK, \ \widetilde\cK$ will be refined below; naturally, we will consider classes of rejection regions that have asymptotically a Type~I error close to the significance level $\alpha$,  similar to the definition of $\cK_\alpha$ from Section~\ref{sec:formulation}. On the other hand, by the same token, one would like to consider rejection regions such that $\bP^{N,T}_{\theta_1}(R_{T})\underset{T\to\infty}\longrightarrow1$.
The concept of asymptotically most powerful test is intended to depict those tests, within the considered class of tests, that have the fastest rate of convergence of their powers to one. Condition \eqref{eq:asym-def} from Definition~\ref{def:mostpoweful}, actually guarantees that $\bP^{N,T}_{\theta_1}(R_{T}^*)$ has the fastest speed of convergence to 1, as $T\to\infty$, among all the elements in class $\mathcal{K}$; respectively \eqref{eq:asym-def-N} implies that $\bP^{N,T}_{\theta_1}(R_{N}^*)$ has the fastest rate of convergence to 1, as $N\to\infty$, among all the elements in class $\widetilde{\cK}$.

Next, we will present the main findings for each asymptotic regime separately, starting with large times.

\subsubsection{Large Times Asymptotics}\label{sec:mainResultsTime}
Throughout this section, we assume that the number of observed Fourier modes $N\in\bN$ is fixed.  \\
We begin by considering the following (asymptotic) class of rejection regions
\begin{align}\label{eq:asym-class}
\mathcal{K}_{\alpha}^*:=\left\{(R_T)_{T\in\bR_+} \, : \, R_T\in\cB(C([0,T]; \bR^N), \  \limsup_{T\to\infty}\bP^{N,T}_{\theta_0}(R_T)\le\alpha\right\}.
\end{align}
The set $\mathcal{K}_{\alpha}^*$ should be seen as a reasonable  asymptotic version of set $\cK_\alpha$.
It consists of test with Type~I error smaller than $\alpha$, in the limit sense.
Apparently, as next result shows, the class $\mathcal{K}_{\alpha}^*$ is `too large' to have likelihood ratio type tests as asymptotically the most powerful.

\begin{theorem}\label{Th:NoMostAsymKstar}
For any positive function of time $c_\alpha(T)$, the rejection region of the form
\begin{align}\label{eq:regRegT4}
R_{T}:=\left\{U_T^N: L(\theta_0,\theta_1,U_T^N)\ge c_{\alpha}(T)\right\},
\end{align}
can not be asymptotically the most powerful in the class $\mathcal{K}_{\alpha}^*$.
\end{theorem}

Towards archiving our main goal, we refine the class $\cK_\alpha$, by considering a slightly smaller class of test $\cK_\alpha^\sharp(\delta)$, defined as follows 
\begin{align}\label{def:NewAsymClass}
\mathcal{K}_{\alpha}^\sharp(\delta):=\left\{(R_T): \limsup_{T\to\infty}\left(\bP^{N,T}_{\theta_0}(R_T)-\alpha\right)\sqrt{T}\le\alpha_1(\delta)\right\},
\quad \delta\in\bR,
\end{align}
where
\begin{equation}\label{eq:alpha1explicit}
\alpha_1(\delta) = (2\pi)^{-1/2}e^{-q_{\alpha}^2/2}\delta + \frac{e^{-q_{\alpha}^2/2}}{2\sqrt{\pi M\theta_0}}\left(\frac{(\theta_1-\theta_0)N}{2(\theta_1+\theta_0)}+1-q_{\alpha}^2\right),
\end{equation}
with $q_\alpha$ denoting the $\alpha$ quantile of a standard Gaussian distribution.

Respectively, for a fixed parameter $\delta\in\bR$, we define the  following family of rejection regions
\begin{align}\label{eq:RTDelta}
R_T^\delta=\left\{U_T^N: L(\theta_0,\theta_1,U_T^N)\ge c^\delta_{\alpha}(T)\right\},
\end{align}
where
\begin{align*}
c^\delta_{\alpha}(T)=\exp\left(-\frac{(\theta_1-\theta_0)^2}{4\theta_0}MT-\frac{\theta_1^2-\theta_0^2} {2\theta_0}\sqrt{\frac{MT}{2\theta_0}}q_{\alpha}-\frac{\delta(\theta_1^2-\theta_0^2)\sqrt{M}}{\sqrt{8\theta_0^3}}\right).
\end{align*}

Note that $R_T^\delta$  indeed is a likelihood ratio type test.
Also note that the class $\mathcal{K}_{\alpha}^\sharp(\delta)$ is  a subset of $\cK_\alpha^*$.
We are still collecting in $\mathcal{K}_{\alpha}^\sharp(\delta)$ tests with Type~I error converging to $\alpha$, but with a rate of convergence at least $\alpha_1T^{-1/2}$, where $\alpha_1$ is a constant independent of $T$. Saying differently, we are accepting Type~I errors which are larger than the significance level $\alpha$ by no more that $\alpha_1(\delta)T^{-1/2} + o(T^{-1/2})$.
The threshold $c_\alpha^\delta(T)$ and the constant $\alpha_1(\delta)$ will be derived naturally over the course of proving the main results, and their particular forms are less important at this point.
We refer the reader to Section~\ref{sec:proofs} for a thorough discussion of the ideas and the methods used to derive these classes of tests. 

\bigskip
Now we are in the position to present one of the main results of this paper.

\begin{theorem}\label{th:MainResult1}
For any $\delta\in\bR$, the rejection region $(R_T^\delta)$ is asymptotically the most powerful in the class $\mathcal{K}_{\alpha}^\sharp(\delta)$.
\end{theorem}

\subsubsection{Asymptotics in large number of Fourier modes}\label{sec:mainResultsSpace}
Now we assume that $T$ is fixed and finite. Taking similar asymptotic approach as in  Section~\ref{sec:mainResultsTime}, we follow the same arguments in considering the corresponding classes of test and the tests themselves, and for brevity we omit the detailed discussion here.

Akin to Theorem~\ref{Th:NoMostAsymKstar}, one can prove the following result.
\begin{theorem}\label{Th:NoMostAsymKtilde} The rejection region $(\widetilde{R}_N)$ of the form
\begin{align}\label{def:RNStar}
\widetilde{R}_N=\left\{U_T^N: L(\theta_0,\theta_1,U_T^N)\ge \widetilde{c}_{\alpha}(N)\right\},
\end{align}
where $\widetilde{c}_\alpha(N)$ is some positive function of $N$, cannot be asymptotically the most powerful in the class
\begin{align}\label{eq:asym-class-N}
\widetilde{\mathcal{K}}_{\alpha}:=\left\{(R_N)_{N\in\bN^+}:  R_T\in\cB(C([0,T]; \bR^N), \ \limsup_{N\to\infty}\bP^{N,T}_{\theta_0}(R_N)\le\alpha\right\}.
\end{align}
\end{theorem}

Similar to $R_T^\delta$, for a fixed parameter $\delta\in\bR$, we define the following family of rejection regions
\begin{align}\label{eq:bestR_N}
\widehat{R}_N^\delta=\left\{U_T^N: L(\theta_0,\theta_1,U_T^N)\ge \widehat{c}^\delta_{\alpha}(N)\right\},
\end{align}
where
\begin{align}\label{eq:bestC_N}
\widehat{c}^\delta_{\alpha}(N)=\exp\left(-\frac{(\theta_1-\theta_0)^2 TM}{4\theta_0}+\frac{(\theta_1-\theta_0)^2N}{8\theta_0^2 }-\frac{\sqrt{TM}(\theta_1^2-\theta_0^2)}{\sqrt{8\theta_0^3}}q_{\alpha} -\frac{\sqrt{T}(\theta_1^2-\theta_0^2)}{\sqrt{8\theta_0^3}}\delta\right).
\end{align}

\smallskip
With this at hand, we present the main result for the case of large number Fourier modes.
\begin{theorem}\label{th:MainResult2}
Assume $\beta/d\ge1/2$. Then, for any $\delta\in\bR$, the rejection region $(\widehat{R}_N^\delta)$ is asymptotically the most powerful in the class
\begin{align}\label{def:Kbar}
\widehat{\mathcal{K}}_{\alpha}(\delta):=\left\{(R_N): \limsup_{N\to\infty}\left(\bP^{N,T}_{\theta_0}(R_N)-\alpha\right)\sqrt{M}\le\widehat{\alpha}_1(\delta)\right\},
\end{align}
where
\begin{align}\label{def:tildeAlpha1}
\widehat{\alpha}_1(\delta)=&\left\{\begin{array}{ccc}\Phi_1^\delta(q_\alpha),&\textrm{ if $\beta/d>1/2$} \\ \Phi_1^\delta(q_\alpha)+\sqrt{\frac{2\beta/d+1}{\varpi^\beta}}\Phi_2^\delta(q_\alpha),&\textrm{ if $\beta/d=1/2$}\end{array}\right.,
\\
\Phi_1^\delta(x)=&\left[\frac{\theta_1-\theta_0}{4\sqrt{\pi\theta_0T}(\theta_1+\theta_0)} \left(\sum_{k=1}^{\infty} e^{-2\theta_0T\lambda_k^{2\beta}}\right)+\frac{1}{2\sqrt{\pi\theta_0T}}(1-x^2)+(2\pi)^{-1/2}\delta\right]e^{-x^2/2},\notag
\\
\Phi_2^\delta(x)=&\frac{(\theta_1-\theta_0)(5\theta_1^2+6\theta_1\theta_0-3\theta_0^2)} {8\sqrt{2\pi}\theta_0(\theta_1+\theta_0)(\theta_1^2-\theta_0^2)T}xe^{-x^2/2},\notag
\end{align}
with $q_\alpha$ denoting the $\alpha$ quantile of a standard Gaussian distribution.
\end{theorem}

\noindent

Although the results are similar to that from previous section, the techniques and methods that deal with asymptotics of the probabilities under null and under alternative are different, as can be seen from the proofs below.

\section{Proofs of the main results}\label{sec:proofs}
In this section we will not only prove the main results over the course of several technical lemmas, but we will also give more details on how the proposed classes of tests introduced in Section~\ref{sec:mainResults} were obtained. We hope that the preliminary discussions from Section~\ref{ssec:AsymptoticT} will lit more light on the structure of those classes of tests, and the idea behind the notion of asymptotically most powerful tests.
Although we included the results on  sharp large deviations bounds (see Lemmas~\ref{lemma:CharExpan-VT}-\ref{lemma:ProbExpan-IT}, and Lemmas~\ref{lemma:CharExpan-VN}-\ref{lemma:ProbExpan-IN}) in this technical part of the paper, we believe they are of independent interest and could be applied to other relevant problems.

\subsection{Large times: proofs}\label{ssec:AsymptoticT}

Throughout this section, we assume that the number of observed Fourier modes $N\in\bN$ is fixed, and we will use the notations from Section~\ref{sec:mainResultsTime}.
We will show that the class $\cK_\alpha^*$, being a natural choice for an asymptotic rejection region at level $\alpha$, it is too large for our purposes,  and finding asymptotically the most powerful tests  of likelihood ratio types within this class is not feasible.
The new class of tests $\cK_\alpha^\sharp(\delta)$, slightly smaller than $\cK_\alpha^*$, allows to identify an easy computable asymptotically the most powerful test.
The reader may wonder why we do consider $\cK_\alpha^*$, and not move directly to the `right' class of test. The reason is twofold.
The series of technical results that lead to Theorem~\ref{th:MorePowerTests}, that show that in $\cK_\alpha^*$ we can find tests that are more powerful than those that are natural candidates, are also essentially used in the proof of the main result - Theorem~\ref{th:MainResult1}. Secondly, while the result is negative for $\cK_\alpha^*$, it gives important insights about the nature of the problem and why $\cK_\alpha^\sharp$ makes sense to be considered. Moreover, and probably most importantly, this gives a better intuition on how to address the case $N\to\infty$, for which most of the technical results are proved quite differently.

Next result is in a sense a version of Neyman-Pearson lemma, that gives sufficient conditions for likelihood ratio type test to be asymptotically the most powerful in the class $\cK_\alpha^*$.

\begin{theorem}\label{th:regRegAT}
Consider the rejection region of the form
\begin{align}\label{eq:regRegT1}
R_{T}^*=\left\{U_T^N: L(\theta_0,\theta_1,U_T^N)\ge c^*_{\alpha}(T)\right\},
\end{align}
where $c^*_{\alpha}(T)$ is a function of $T$ such that, $c^*_{\alpha}(T)>0$ for all $T>0$ and
\begin{align}
&\lim_{T\to\infty}\bP^{N,T}_{\theta_0}(R_{T}^*) = \alpha, \label{eq:RegT2} \\
&\lim_{T\to\infty}\frac{c^*_{\alpha}(T)}{1-\bP^{N,T}_{\theta_1}(R_{T}^*)}<\infty. \label{eq:RegT3}
\end{align}
Then $(R_{T}^*)$ is asymptotically the most powerful in $\mathcal{K}_{\alpha}^*$.
\end{theorem}
\begin{proof}
Assume that $(R^*_T)$ satisfies \eqref{eq:regRegT1}-\eqref{eq:RegT3}. By the same reasoning as in Theorem~\ref{th:NPLemma}, for a fixed $T$ and any $(R_T)\in\mathcal{K}_{\alpha}^*$, we have that
\begin{align*}
\bP^{N,T}_{\theta_1}(R_{T}^*)-\bP^{N,T}_{\theta_1}(R_T)\ge c^*_{\alpha}(T)\left(\bP^{N,T}_{\theta_0}(R_{T}^*)-\bP^{N,T}_{\theta_0}(R_T)\right),
\end{align*}
which can be written as
\begin{align*}
\frac{1-\bP^{N,T}_{\theta_1}(R_T)}{1-\bP^{N,T}_{\theta_1}(R_{T}^*)} \ge1+\frac{c^*_{\alpha}(T)}{1-\bP^{N,T}_{\theta_1}(R_{T}^*)}\left(\bP^{N,T}_{\theta_0}(R_{T}^*)-\bP^{N,T}_{\theta_0}(R_T)\right).
\end{align*}
From here, using \eqref{eq:RegT2} and \eqref{eq:RegT3}, we deduce
\begin{align*}
\underset{T\to\infty}{\liminf}\frac{1-\bP^{N,T}_{\theta_1}(R_T)}{1-\bP^{N,T}_{\theta_1}(R_{T}^*)} \ge&1+\lim_{T\to\infty}\frac{c^*_{\alpha}(T)}{1-\bP^{N,T}_{\theta_1}(R_{T}^*)}\lim_{T\to\infty}\bP^{N,T}_{\theta_0}(R_{T}^*) \\
&-\lim_{T\to\infty}\frac{c^*_{\alpha}(T)}{1-\bP^{N,T}_{\theta_1}(R_{T}^*)}\limsup_{T\to\infty}\bP^{N,T}_{\theta_0}(R_{T}) \\
=&1+\lim_{T\to\infty}\frac{c^*_{\alpha}(T)}{1-\bP^{N,T}_{\theta_1}(R_{T}^*)}\left(\alpha-\limsup_{T\to\infty}\bP^{N,T}_{\theta_0}(R_{T})\right)\ge1.
\end{align*}
This completes the proof.
\end{proof}

Of course the goal is to find an explicit formula for $c^*_\alpha(T)$ such that \eqref{eq:regRegT1}-\eqref{eq:RegT3}  are satisfied.
We will start with the following  heuristic arguments.
By \eqref{eq:OU-Fourier}, \eqref{eq:RadonNikodymUn},  and It\^o's formula we have the following
\begin{align}\label{eq:heuristT}
\bP^{N,T}_{\theta_0}& (L(\theta_0,\theta_1,U_T^N)\ge c^*_{\alpha})\notag \\
=&\bP^{N,T}_{\theta_0}\left(-\sum_{k=1}^N\lambda_k^{2\beta+2\gamma}\left(\int_0^Tu_k(t)du_k(t) +\frac{1}{2}(\theta_1+\theta_0)\lambda_k^{2\beta}\int_0^Tu_k^2(t)dt\right)\ge \frac{\sigma^{2}\ln c^*_{\alpha}}{\theta_1-\theta_0}\right)\notag
\\
=&\bP^{N,T}_{\theta_0}\left(\sum_{k=1}^N\lambda_k^{2\beta+2\gamma}\left(\frac{\theta_1-\theta_0}{2}\left(u_k^2(T)-\sigma^2\lambda_k^{-2\gamma}T\right)-(\theta_1+\theta_0)\sigma\lambda_k^{-\gamma}\int_0^Tu_kdw_k\right)\ge \frac{2\theta_0\sigma^{2}\ln c^*_{\alpha}}{\theta_1-\theta_0}\right)\notag
\\
= &\bP^{N,T}_{\theta_0}\left(X_T-\frac{2(\theta_1+\theta_0)}{(\theta_1-\theta_0)\sigma\sqrt{T}}Y_T\ge \frac{4\theta_0\ln c^*_{\alpha}}{(\theta_1-\theta_0)^2T}+M\right),
\end{align}
where
\begin{align*}
M:=\sum_{k=1}^N\lambda_k^{2\beta}, \quad X_T:=\sum_{k=1}^N\frac{\lambda_k^{2\beta+2\gamma}u_k^2(T)}{\sigma^2T},\qquad Y_T:=\frac{1}{\sqrt{T}}\sum_{k=1}^N\lambda_k^{2\beta+\gamma}\int_0^Tu_kdw_k.
\end{align*}
Next note that, since $X_T\ge0$, we have that
\begin{align}\label{15}
\bP^{N,T}_{\theta_0}(L(\theta_0,\theta_1,U_T^N)\ge c^*_{\alpha})\ge\bP^{N,T}_{\theta_0}\left(-\frac{2(\theta_1+\theta_0)}{(\theta_1-\theta_0)\sigma\sqrt{T}}Y_T\ge \frac{4\theta_0\ln c^*_{\alpha}}{(\theta_1-\theta_0)^2T}+M\right).
\end{align}
By \eqref{eq:7}, we get that  $Y_T \overset{d}\rightarrow\mathcal{N}(0,\sigma^2M/(2\theta_0))$, as $T\to\infty$. Thus, it is reasonable to choose $c_{\alpha}$ such that
\begin{align*}
-\sqrt{\frac{2\theta_0}{M}}\frac{(\theta_1-\theta_0)\sqrt{T}}{2(\theta_1+\theta_0)}\left[\frac{4\theta_0\ln c^*_{\alpha}}{(\theta_1-\theta_0)^2T}+M\right]= q_\alpha,
\end{align*}
where $q_\alpha$ is $\alpha$ quantile of standard Gaussian distribution. Hence, we take
\begin{align}\label{eq:cAlpha1}
c_{\alpha}^\sharp(T)=\exp\left(-\frac{(\theta_1-\theta_0)^2}{4\theta_0}MT-\frac{\theta_1^2-\theta_0^2}{2\theta_0}\sqrt{\frac{MT}{2\theta_0}}q_{\alpha}\right).
\end{align}
Henceforth, we will denote the rejection region $(R_T^*)$ given by \eqref{eq:regRegT1} and the above $c_\alpha^\sharp$ by $(R_T^\sharp)$, that is,
\begin{align}\label{eq:RTSharp}
R_T^\sharp:=\{U_T^N: L(\theta_0,\theta_1,U_T^N)\ge c_{\alpha}^\sharp(T)\},\qquad\textrm{for all $T$},
\end{align}
where $c_{\alpha}^\sharp(T)$ is defined by \eqref{eq:cAlpha1}.

\begin{proposition}\label{prop:RTstar-in-K}
The rejection region $(R_T^\sharp)$ belongs to $\mathcal{K}_{\alpha}^*$, and moreover
\begin{align*}
\lim_{T\to\infty}\bP^{N,T}_{\theta_0}(R_{T}^\sharp)=\alpha.
\end{align*}
\end{proposition}

\begin{proof}
By \eqref{eq:heuristT}, we deduce that, for any $\delta>0$,
\begin{align}\label{eq:19}
\bP^{N,T}_{\theta_0}&(L(\theta_0,\theta_1,U_T^N)\ge c_{\alpha}^\sharp)\le\bP^{N,T}_{\theta_0}\left(X_T\ge\delta\right)
+\bP^{N,T}_{\theta_0}\left(-\frac{2(\theta_1+\theta_0)}{(\theta_1-\theta_0)\sigma\sqrt{T}}Y_T\ge \frac{4\theta_0\ln c_{\alpha}^\sharp}{(\theta_1-\theta_0)^2T}+M-\delta\right).
\end{align}
Next, taking into account that $u_k(T)\overset{d}{\sim}\mathcal{N}\left(0,\sigma^2\displaystyle\frac{1-\exp(-2\lambda_k^{2\beta}\theta_0 T)}{2\lambda_k^{2\beta+2\gamma}\theta_0}\right)$, we obtain the following estimates
\begin{align}\label{eq:20}
\bP^{N,T}_{\theta_0}\left(X_T\ge\delta\right)=&\bP^{N,T}_{\theta_0}\left(\sum_{k=1}^N\frac{\lambda_k^{2\beta+2\gamma}u_k^2(T)}{\sigma^2T}\ge\delta\right)
\le\sum_{k=1}^N\bP^{N,T}_{\theta_0}\left(\frac{\lambda_k^{2\beta+2\gamma}u_k^2(T)}{\sigma^2T}\ge\frac{\delta}{N}\right)\notag \\
=&\sum_{k=1}^N\bP^{N,T}_{\theta_0}\left(|u_k(T)|\ge\sqrt{\frac{\sigma^2\delta T}{\lambda_k^{2\beta+2\gamma}N}}\right)
=2\sum_{k=1}^N\Phi\left(-\sqrt{\frac{2\delta\theta_0 T}{N\left(1-\exp\left(-2\lambda_k^{2\beta}\theta_0 T\right)\right)}}\right)\notag\\
\le&2N\Phi\left(-\sqrt{\frac{2\delta\theta_0 T}{N}}\right)\le\sqrt{\frac{4N^3}{\pi\delta\theta_0 T}}\exp(-\delta\theta_0 T/N),
\end{align}
where  $\Phi(\cdot)$ denotes the cumulative distribution function of a standard Gaussian.
Using \eqref{15} with $c_{\alpha}^*=c_{\alpha}^\sharp$ and asymptotic normality of $Y$, we conclude that $\underset{T\to\infty}{\liminf}\bP^{N,T}_{\theta_0}(R_{T}^\sharp)\ge\alpha.$
Taking in \eqref{eq:19} and \eqref{eq:20}, $\delta=\displaystyle\frac{2(\theta_1+\theta_0)\epsilon}{(\theta_1-\theta_0)\sigma\sqrt{T}}$, with $\epsilon>0$,  we obtain
\begin{align*}
\bP^{N,T}_{\theta_0}(R_{T}^\sharp)=&\bP^{N,T}_{\theta_0}(L(\theta_0,\theta_1,U_T^N)\ge c_{\alpha}^\sharp) \\
\le&\sqrt{\frac{2N^3(\theta_1+\theta_0)\sigma}{\pi\epsilon\theta_0(\theta_1+\theta_0)\sqrt{T}}}
\exp\left(-\frac{2\epsilon\theta_0(\theta_1+\theta_0)\sqrt{T}}{(\theta_1-\theta_0)N\sigma}\right) \\
&\quad+\bP^{N,T}_{\theta_0}(Y_T\le\sqrt{\sigma^2M/(2\theta_0)}q_{\alpha}+\epsilon).
\end{align*}
Finally, taking the limsup in the above inequality, and using asymptotic normality of $Y$, and using the fact that $\epsilon$ is arbitrary, we conclude that
\begin{align*}
\limsup_{T\to\infty}\bP^{N,T}_{\theta_0}(R_T^\sharp)\le\alpha.
\end{align*}
This ends the proof.
\end{proof}

While $(R_T^\sharp)$ has exact asymptotic level $\alpha$ and belongs to $\cK_\alpha^*$, unfortunately it may not necessarily be asymptotically the most powerful in $\cK_\alpha^*$.
More precisely,  it may not satisfy \eqref{eq:RegT3}, and eventually we will show that one can find many other tests that converge to one faster than $(R_T^\sharp)$.
However, it `almost' satisfies the desired properties of convergence, and by shrinking appropriately the class $\mathcal{K}^*_{\alpha}$, we can make a test similar to $(R_T^\sharp)$ to be the most powerful in the new class. To achieve this, we will need some fine results on the asymptotics of the power of the test  $\bP_{\theta_1}(R^\sharp_T)$ which are presented in the following series of lemmas. In particular, we make use of some new results on large deviation principle of the Log-likelihood ratio and find its rate function.

\subsubsection{Cumulant generating function of the Log-Likelihood ratio}\label{ssec:CumFuncT}

 We start with computation of the cumulant generating function $m_T(\epsilon)=\bE\left[\exp\left(\epsilon\ln L(\theta_0,\theta_1,U_T^N)\right)\right]$ (see Appendix~\ref{appendix:cum}) of the Log-Likelihood ratio. To accomplish this, we will use appropriately Feynman--Kac formula (for a similar approach, see also \citet{GapeevKuchler2008}).

By It\^o's formula, we get
\begin{align}\label{eq:11}
\ln L(\theta_0,\theta_1,U_T^N)=
-\frac{\theta_1-\theta_0}{2\sigma^2}\sum_{k=1}^N\lambda_k^{2\beta+2\gamma}u_k^2(T)
+\frac{\theta_1-\theta_0}{2}MT
-\frac{\theta_1^2-\theta_0^2}{2\sigma^2}\sum_{k=1}^N\int_0^T\lambda_k^{4\beta+2\gamma}u_k^2(t)dt.
\end{align}
For $\mathbf{u}=(u_1,\ldots,u_N)\in\bR^N$ and $\epsilon>0$, we set
\begin{align*}
r(\mathbf{u},t)&:=\frac{\epsilon(\theta_1^2-\theta_0^2)}{2\sigma^2}\sum_{k=1}^N\lambda_k^{4\beta+2\gamma}u_k^2, \\
F(\mathbf{u}) & := \exp\left(-\frac{\epsilon(\theta_1-\theta_0)}{2\sigma^2}\sum_{k=1}^N\lambda_k^{2\beta+2\gamma}u_k^2
                +\frac{\epsilon(\theta_1-\theta_0)}{2}MT\right).
\end{align*}
By Feynman-Kac formula, the function
\begin{align}\label{eq:12}
f(\mathbf{u},t)=\bE_{\theta_1}\left[\exp\left(-\int_t^Tr(U_s^N,s)ds\right)F(U_T^N)\bigg|U_t^N=\mathbf{u}\right]\qquad\qquad\quad
\end{align}
is the only solution to the following PDE
\begin{align*}
f_t+\frac{\sigma^2}{2}\sum_{k=1}^N\lambda_k^{-2\gamma}f_{u_ku_k}-\theta_1\sum_{k=1}^N\lambda_k^{2\beta}u_kf_{u_k}=r(\mathbf{u},t) f,\qquad f(\mathbf{u},T)=F(\mathbf{u}).
\end{align*}
By making the transformation
$$
f=\exp\left(\frac{\theta_1}{2}\sum_{k=1}^N\lambda_k^{2\beta}\left(\sigma^{-2}\lambda_k^{2\gamma}u_k^2-t\right)\right)
g\left(t,\sigma^{-1}\textrm{diag}[\lambda_1^\gamma,\ldots,\lambda_n^\gamma]\mathbf{u}\right),
$$
we obtain
\begin{align*}
g_t+\frac{1}{2}\Delta_{\mathbf{u}}g=&\frac{\epsilon(\theta_1^2-\theta_0^2)+\theta_1^2}{2}\left(\sum_{k=1}^N\lambda_k^{4\beta}u_k^2\right)g, \\
g(\mathbf{u},T)=&\exp\left(-\frac{\epsilon(\theta_1-\theta_0)+\theta_1}{2}\sum_{k=1}^N\lambda_k^{2\beta}u_k^2+\frac{\epsilon(\theta_1-\theta_0)+\theta_1}{2}MT\right).
\end{align*}
As it turns out, this PDE can be solved explicitly (with a good initial guess of the form of the solution), and consequently, we get an explicit formula for $f$,
\begin{align}\label{eq:13}
f(\mathbf{u},t)=&\exp\left(\sum_{k=1}^N\alpha_ku_k^2\left[\sinh(\gamma_kt)
                +\beta_k\cosh(\gamma_kt)\right]\left[\cosh(\gamma_kt)+\beta_k\sinh(\gamma_kt)\right]^{-1}\right.\notag \\
&\left.+\frac{1}{2}\sum_{k=1}^N\ln\left|\frac{\cosh(\gamma_kT)+\beta_k\sinh(\gamma_kT)}{\cosh(\gamma_kt)+\beta_k\sinh(\gamma_kt)}\right| +\frac{\theta_1}{2\sigma^2}\sum_{k=1}^N\lambda_k^{2\beta+2\gamma}u_k^2+\frac{\epsilon(\theta_1-\theta_0)+\theta_1}{2}MT-\frac{\theta_1}{2}Mt\right),
\end{align}
where
\begin{align*}
\alpha_k=&-\sigma^{-2}\left(\epsilon(\theta_1^2-\theta_0^2)+\theta_1^2\right)^{1/2}\lambda_k^{2\beta+2\gamma}/2, &&\gamma_k=-\left(\epsilon(\theta_1^2-\theta_0^2)+\theta_1^2\right)^{1/2}\lambda_k^{2\beta}, \\
\beta_k=&\frac{p\cosh(\gamma_kT)-\sinh(\gamma_kT)}{\cosh(\gamma_kT)-p\sinh(\gamma_kT)}, &&p=\frac{\epsilon(\theta_1-\theta_0)+\theta_1}{\left(\epsilon(\theta_1^2-\theta_0^2)+\theta_1^2\right)^{1/2}}.
\end{align*}
By taking $\mathbf{u}=t=0$ in \eqref{eq:12},  and using \eqref{eq:13}, we obtain
\begin{align}\label{eq:34}
m_T(\epsilon)=&\bE_{\theta_1}\left[\exp\left(\epsilon\ln L(\theta_0,\theta_1,U_T^N)\right)\right]=f(0,0)\notag
\\
=&\exp\left[-\frac{1}{2}\sum_{k=1}^N\ln\left(\cosh(\gamma_kT)-p\sinh(\gamma_kT)\right)+\frac{\epsilon(\theta_1-\theta_0)+\theta_1}{2}MT\right].
\end{align}
Note that,
\begin{align*}
c(\epsilon):=\lim_{T\to\infty}T^{-1}\ln m_T(\epsilon)=\left(-\left(\epsilon(\theta_1^2-\theta_0^2)+\theta_1^2\right)^{1/2}+\epsilon(\theta_1-\theta_0)+\theta_1\right)\frac{M}{2},
\end{align*}
and, also it is easy to see that $c(\epsilon)$ is proper and convex with\footnote{See Definition~\ref{def:condM} for notations $\epsilon_\pm,\gamma_\pm,\gamma_0$. }
\begin{align*}
\epsilon_-=&-\frac{\theta_1^2}{\theta_1^2-\theta_0^2},\qquad\epsilon_+=+\infty, \\
\eta_-=&-\infty,\qquad\eta_0=\frac{(\theta_1-\theta_0)^2}{4\theta_1}M,\qquad\eta_+=\frac{\theta_1-\theta_0}{2}M.
\end{align*}
Thus, $m_t(\varepsilon)$ satisfies Condition~\textbf{(m)}, Definition~\ref{def:condM}, with $\varphi_T=T$.

Having these, and by applying Theorem~\ref{th:Chernoff}, we get immediately   the following result
\begin{proposition}\label{prop:LargDev1}
The Log-Likelihood function satisfies the following identities
\begin{align}
&\lim_{T\to\infty}T^{-1}\ln \bP\left(T^{-1}\ln L(\theta_0,\theta_1,U_T^N)\geq\eta\right)=-I(\eta), &&\eta\in\left(\frac{(\theta_1-\theta_0)^2}{4\theta_1}M,\frac{\theta_1-\theta_0}{2}M\right),\notag
\\
&\lim_{T\to\infty}T^{-1}\ln \bP\left(T^{-1}\ln L(\theta_0,\theta_1,U_T^N)\leq\eta\right)=-I(\eta), &&\eta\in\left(-\infty,\frac{(\theta_1-\theta_0)^2}{4\theta_1}M\right),\label{eq:22}
\end{align}
where $I$ is the Legendre-Fenchel transform that takes the following form
\begin{align}\label{eq:23}
I(\eta)=\underset{\epsilon>\epsilon_-}{\sup}(\epsilon\eta-c(\epsilon)) =\left\{\begin{array}{ccc}\displaystyle-\frac{(4\theta_1\eta-(\theta_1-\theta_0)^2M)^2}{8(2\eta-(\theta_1-\theta_0)M) (\theta_1^2-\theta_0^2)},&\eta<\displaystyle\frac{\theta_1-\theta_0}{2}M\\+\infty,&\eta\ge\displaystyle\frac{\theta_1-\theta_0}{2}M\end{array}\right..
\end{align}
\end{proposition}

Using the above result, we get the rates of convergence of the power of the test.
\begin{theorem}\label{th:asympT1-simple}
\begin{align*}
\bP^{N,T}_{\theta_1}(R_T^\sharp)=1-\exp\left(-\frac{(\theta_1-\theta_0)^2}{4\theta_0}MT+o(T)\right).
\end{align*}
\end{theorem}
\begin{proof}
It is equivalent to show that
$$
\lim_{T\to\infty}T^{-1}\ln\left(1-\bP^{N,T}_{\theta_1}(R_T^\sharp)\right)=-\frac{(\theta_1-\theta_0)^2}{4\theta_0}M.
$$
By the definition we have
\begin{align*}
1-\bP^{N,T}_{\theta_1}(R_T^\sharp)=&1-\bP^{N,T}_{\theta_1}\left(\ln L(\theta_0,\theta_1,U_T^N)\ge\ln c_{\alpha}^\sharp\right) \\
=&\bP^{N,T}_{\theta_1}\left(T^{-1}\ln L(\theta_0,\theta_1,U_T^N)\le-\frac{(\theta_1-\theta_0)^2}{4\theta_0}M
-\frac{\theta_1^2-\theta_0^2}{2\theta_0}\sqrt{\frac{M}{2\theta_0T}}q_{\alpha}\right).
\end{align*}
Note that for small $\alpha$, $q_{\alpha}<0$, and thus, using \eqref{eq:22} and \eqref{eq:23}, we deduce
\begin{align*}
\liminf_{T\to\infty}T^{-1}\ln\left(1-\bP^{N,T}_{\theta_1}(R_T^\sharp)\right)
&\geq \lim_{T\to\infty}T^{-1}\ln\bP^{N,T}_{\theta_1}\left(T^{-1}\ln L(\theta_0,\theta_1,U_T^N)\le-\frac{(\theta_1-\theta_0)^2}{4\theta_0}M\right)\\
& =-\frac{(\theta_1-\theta_0)^2}{4\theta_0}M.
\end{align*}
On the other hand, for any $\delta>0$, we have that
\begin{align*}
\limsup_{T\to\infty}T^{-1}\ln\left(1-\bP^{N,T}_{\theta_1}(R_T^\sharp)\right)
\le&\lim_{T\to\infty}T^{-1}\ln\bP^{N,T}_{\theta_1}\left(T^{-1}\ln L(\theta_0,\theta_1,U_T^N)\le-\frac{(\theta_1-\theta_0)^2}{4\theta_0}M+\delta\right) \\
=&-\frac{((\theta_1^2-\theta_0^2)(\theta_1-\theta_0)M/\theta_0-4\theta_1\delta)^2}{4(\theta_1^2-\theta_0^2)((\theta_1^2-\theta_0^2)M/\theta_0-4\delta)}.
\end{align*}
Passing to the limit in the last inequality with $\delta\to0^+$, we get
\begin{align*}
\limsup_{T\to\infty}T^{-1}\ln\left(1-\bP^{N,T}_{\theta_1}(R_T^\sharp)\right)\le-\frac{(\theta_1-\theta_0)^2}{4\theta_0}M.
\end{align*}
The theorem is proved.
\end{proof}

\subsubsection{Sharp large deviation principle}\label{sec:sharpLargeDevT}
Theorem~\ref{th:asympT1-simple} essentially implies that the probability of $R_T^\sharp$ under $\mathscr{H}_1$ goes exponentially fast to 1, as $T\to\infty$, with the rate of convergence $\frac{(\theta_1-\theta_0)^2}{4\theta_0}M$.
However, this result is not sufficient to answer whether $(R_T^\sharp)$ has the highest asymptotic power, for which we need more precise convergence results (i.e. of the higher order term $o(T)$). Next, we will introduce some results on sharp large deviations for the stochastic processes relevant to this study. The ideas are similar to those from \citet{BercuRouault2001}.
The aim is to extract the exponential part from the asymptotics of Type~II error. Namely, for the function
$$
\mathcal{L}_T(\epsilon)  =T^{-1}\ln\bE_{\theta_1}\left[\exp\left(\epsilon\ln L(\theta_0,\theta_1,U_T^N)\right)\right],
$$
we are looking for decomposition of the form
$$
\mathcal{L}_T(\epsilon)  = \mathcal{L}(\epsilon)+T^{-1}\mathcal{H}(\epsilon)+T^{-1}\mathcal{R}_T(\epsilon),
$$
with some appropriate functions $\mathcal{L}(\epsilon),  \mathcal{H}(\epsilon), \mathcal{R}_T(\epsilon)$ derived below.

Let us consider the Ornstein-Uhlenbeck process
\begin{align*}
d u_k(t) = -\theta_1 \lambda_k^{2\beta} u_k(t)d t + \sigma \lambda_k^{-\gamma} d w_k(t), \quad  u_k(0) = 0.
\end{align*}
 In what follows, we will use notations similar to those from Appendix~\ref{appendix:SLDforOU} with
\begin{align*}
a_k&:=-\epsilon(\theta_1-\theta_0)\lambda_k^{2\beta+2\gamma}/\sigma^2, \\
b_k& : =-\epsilon(\theta_1^2-\theta_0^2)\lambda_k^{4\beta+2\gamma}/2\sigma^2,
\end{align*}
and define
\begin{align}
\mathcal{L}^k(\epsilon)&:=\frac{\lambda_k^{2\beta}}{2}\left(\theta_1+(\theta_1-\theta_0)\epsilon-\sqrt{\theta_1^2+(\theta_1^2-\theta_0^2)\epsilon}\right),\nonumber\\
\mathcal{D}(\epsilon)&:=\frac{\theta_1+(\theta_1-\theta_0)\epsilon}{\sqrt{\theta_1^2+(\theta_1^2-\theta_0^2)\epsilon}}, \
\mathcal{H}^k(\epsilon):=-\frac{1}{2}\ln\left(\frac{1}{2}+\frac{1}{2}\mathcal{D}(\epsilon)\right), \nonumber\\
\mathcal{R}^k_T(\epsilon)&:=-\frac{1}{2}\ln\left(1+\frac{1-\mathcal{D}(\epsilon)}{1+\mathcal{D}(\epsilon)}
\exp\left(-2\lambda_k^{2\beta}T\sqrt{\theta_1^2+(\theta_1^2-\theta_0^2)\epsilon}\right)\right),  \label{eq:notationLHR}\\
\cL(\epsilon)&: =\sum_{k=1}^N \cL^k(\epsilon), \
\cH(\epsilon): =\sum_{k=1}^N \cH^k(\epsilon), \
\cR_T(\epsilon): =\sum_{k=1}^N \cR_T^k(\epsilon). \nonumber
\end{align}
Using the independence of $u_k$'s and Proposition~\ref{lemma:BercueSLD}, we have
\begin{align}\label{MomentLLRSplit}
\mathcal{L}_T(\epsilon) & :=T^{-1}\ln\bE_{\theta_1}\left[\exp\left(\epsilon\ln L(\theta_0,\theta_1,U_T^N)\right)\right]\notag \\
&=T^{-1}\sum_{k=1}^N\ln\bE_{\theta_1}\left[\exp\left(\mathcal{Z}_T(a_k,b_k)\right)\right]\notag \\
&=\sum_{k=1}^N\left(\mathcal{L}^k(\epsilon)+T^{-1}\mathcal{H}^k(\epsilon)+T^{-1}\mathcal{R}^k_T(\epsilon)\right)\notag\\
&=\mathcal{L}(\epsilon)+T^{-1}\mathcal{H}(\epsilon)+T^{-1}\mathcal{R}_T(\epsilon).
\end{align}
We also set
\begin{align}\label{eq:29}
A_T & :=\exp\left[T(\mathcal{L}_T(\epsilon_{\eta})-\eta\epsilon_{\eta})\right],\notag \\
B_T & :=\bE_T\left(\exp\left[-\epsilon_{\eta}(\ln L(\theta_0,\theta_1,U_T^N)-\eta T)\right]\1_{\{\ln L(\theta_0,\theta_1,U_T^N)\le\eta^* T\}}\right),
\end{align}
where $\eta$ and $\eta^*$ are some numbers which may depend on $T$, and $\bE_T$ is the expectation under $\bQ_T$ with
\begin{align}\label{eq:MeasureChange1}
\frac{d\mathbb{Q}_T}{d\bP^{N,T}_{\theta_1}}=\exp\left(\epsilon_{\eta}\ln L(\theta_0,\theta_1,U_T^N)-T\mathcal{L}_T(\epsilon_{\eta})\right).
\end{align}
Clearly,
$$
\bP^{N,T}_{\theta_1}\left(\ln L(\theta_0,\theta_1,U_T^N)\le\eta^* T\right)=A_TB_T.
$$
Naturally, by taking $\epsilon_{\eta}$ such that $\mathcal{L}'(\epsilon_{\eta})=\eta$, we get
\begin{align}\label{eq:epsilonN}
\epsilon_{\eta}=\frac{(\theta_1^2-\theta_0^2)^2M^2-4\theta_1^2(-2\eta+(\theta_1-\theta_0)M)^2} {4(\theta_1^2-\theta_0^2)(-2\eta+(\theta_1-\theta_0)M)^2},
\end{align}
and then by direct computations we find
\begin{align}\label{eq:ATVariate}
A_T=&\exp\left[T(\mathcal{L}(\epsilon_{\eta})-\eta\epsilon_{\eta})\right]
\exp\left[\mathcal{H}(\epsilon_{\eta})+\mathcal{R}_T(\epsilon_{\eta})\right]\notag\\
=&\exp\left(-I(\eta)T\right)\left(\frac{1}{2}+\frac{1}{2}\mathcal{D}(\epsilon_{\eta})\right)^{-N/2}
\prod_{k=1}^N\left(1+\frac{1-\mathcal{D}(\epsilon_{\eta})}{1+\mathcal{D}(\epsilon_{\eta})}
\exp\left(-2\lambda_k^{2\beta}T\sqrt{\theta_1^2+(\theta_1^2-\theta_0^2)\epsilon_{\eta}}\right)\right)^{-1/2},
\end{align}
where $I(\eta)$ is given by \eqref{eq:23}.

\begin{lemma}\label{lemma:CharExpan-VT}
Under probability measure $\bQ_T$, the following holds true
\begin{align*}
V_T:=\frac{\ln L(\theta_0,\theta_1,U_T^N)-\eta T}{\varsigma_{\eta}\sqrt{T}}\overset{d}{\longrightarrow}\mathcal{N}(0,1), \quad T\to\infty,
\end{align*}
where $\eta<\frac{(\theta_1-\theta_0)^2}{4\theta_1}M$, and
\begin{equation*}
\varsigma_{\eta}^2:=\mathcal{L}''(\epsilon_{\eta})=\frac{(-2\eta+(\theta_1-\theta_0)M)^3}{(\theta_1^2-\theta_0^2)M^2}.
\end{equation*}
More precisely, for any integer $n>0$, there exist integers $s_1(n)$, $s_2(n)$, $s_3(n)$, and a sequence $(\varphi_{k,l})$ independent of $n$, such that, for large enough $T$, the characteristic function of $V_T$ under measure $\bQ_T$, denoted by $\Psi_{V_T}$, has the following expansion
\begin{align}\label{eq:CharExpan-VT}
\Psi_{V_T}(u)=\exp\left(-\frac{u^2}{2}\right)\left[1+\frac{1}{\sqrt{T}}\sum_{k=0}^{n}\sum_{l=k+1}^{s_1(n)} \frac{\varphi_{k,l}u^l}{(\sqrt{T})^k}+O\left(\frac{|u|^{s_2(n)}+|u|^{s_3(n)}}{T^{(n+2)/2}}\right)\right].
\end{align}
Moreover, the remainder is uniformly bounded in $u$ as long as $|u|=O(T^{1/6})$\footnote{That is, there exists $M\in\bR$, such that  $\left|\frac{O\left(\frac{|u|^{s_2(n)}+|u|^{s_3(n)}}{T^{(n+2)/2}}\right)}{\frac{|u|^{s_2(n)}+|u|^{s_3(n)}}{T^{(n+2)/2}}}\right|\leq M$, for $\frac{|u|}{T^{1/6}}\leq c$ with $c\in\bR$.}.
\end{lemma}
\begin{proof}
We follow the same lines of proof as in \citet[Lemma 6.1 and Lemma 7.1]{BercuRouault2001}.  From \eqref{eq:MeasureChange1} we immediately have
\begin{align}\label{eq:Char-VT}
\Psi_{V_T}(u)=\exp\left[-\frac{iu\eta\sqrt{T}}{\varsigma_{\eta}}+T\left(\mathcal{L}_T\left(\epsilon_{\eta} +\frac{iu}{\varsigma_{\eta}\sqrt{T}}\right)-\mathcal{L}_T(\epsilon_{\eta})\right)\right].
\end{align}
It follows from \eqref{eq:notationLHR} that for any $k\in\mathbb{N}$, $\mathcal{R}_T^{(k)}(\epsilon_{\eta})=O\left(T^{k-1}e^{-CT}\right)$, for some constant $C$.
Thereby, using \eqref{MomentLLRSplit}, we get
\begin{align}\label{eq:Deriv-LLR}
\mathcal{L}_T^{(k)}(\epsilon_{\eta})=\mathcal{L}^{(k)}(\epsilon_{\eta})+T^{-1}\mathcal{H}^{(k)}(\epsilon_{\eta})+O\left(T^ke^{-CT}\right).
\end{align}
Consequently, via \eqref{eq:Char-VT} together with \eqref{eq:Deriv-LLR}, we have the following expansion
\begin{align}\label{VCharLogExpan}
\ln\Psi_{V_T}(u)=-\frac{u^2}{2}+T\sum_{k=3}^{n+3}\left(\frac{iu}{\varsigma_{\eta}\sqrt{T}}\right)^k\frac{\mathcal{L}^{(k)}(\epsilon_{\eta})}{k!} +\sum_{k=1}^{n+1}\left(\frac{iu}{\varsigma_{\eta}\sqrt{T}}\right)^k\frac{\mathcal{H}^{(k)}(\epsilon_{\eta})}{k!} +O\left(\frac{|u|^{n+2}+|u|^{n+4}}{T^{(n+2)/2}}\right).
\end{align}
From here,  we follow the same approach as in \citet[Lemma 2 p.72]{Cramer1970}, and we obtain \eqref{eq:CharExpan-VT}.
Moreover, from the proof it follows that $u^l/(\sqrt{T})^k$ is bounded, provided that $|u|=O(T^{1/6})$, and consequently the last part of the Lemma follows.
\end{proof}

Next we provide an asymptotic result for the term $B_T$.
\begin{lemma}\label{lemma:BT-AsymEst} For
\begin{align*}
\eta=-\frac{(\theta_1-\theta_0)^2}{4\theta_0}M,\qquad \eta^*=-\frac{(\theta_1-\theta_0)^2}{4\theta_0}M-\frac{\theta_1^2-\theta_0^2}{2\theta_0}\sqrt{\frac{M}{2\theta_0T}}q_{\alpha},
\end{align*}
as $T\to\infty$, the following asymptotic holds true
\begin{align*}
B_T\sim \exp\left(-\frac{\theta_1^2-\theta_0^2}{2\theta_0}\sqrt{\frac{MT}{2\theta_0}}q_{\alpha}\right)/\sqrt{T}.
\end{align*}
\end{lemma}
\begin{proof} Similar to the notations in the Appendix~\ref{appendix:SLDforOU} and those from \citet{BercuRouault2001}, we put
$$
\mathcal{Z}_T^k := -\sigma^{-2}(\theta_1-\theta_0)\lambda_k^{2\beta+2\gamma}\int_0^Tu_k(t)du_k(t) -\frac{\sigma^{-2}}{2}(\theta_1^2-\theta_0^2)\lambda_k^{4\beta+2\gamma}\int_0^Tu_k^2(t)dt, \quad k=1,\ldots,N.
$$
Note that the results from  \citet{BercuRouault2001} hold true for each $\cZ_T^k$, and one can derive similar results for $\ln L(\theta_0,\theta_1,U_T^N)=\sum_{k=1}^N\cZ_T^k$. In particular, using \eqref{eq:Char-VT}, by similar\footnote{the evaluations are indeed similar, but lengthy, and for the sake of space conservation we omit them here.} evaluations as in \citet[Lemma~6.1]{BercuRouault2001}, we obtain, for $T$ large enough, the following estimate,
\begin{align}\label{eq:CharExControl}
|\Psi_{V_T}(u)| \le\left(1+\frac{\mu_1 u^2}{T}\right)^{-\mu_2 T},
\end{align}
where $\mu_1$,  $\mu_2$ are some positive constants.

Therefore, for large enough $T$, $\Psi_{V_T}\in L^2(\mathbb{R})$, and using the inverse Fourier transform for $\Psi_{V_T}$, we derive the following equalities
\begin{align*}
B_T=&\bE_T\left(\exp\left(-\epsilon_{\eta}\varsigma_{\eta}\sqrt{T}V_T\right) \1_{\{V_T\le\frac{\eta^*-\eta}{\varsigma_{\eta}}\sqrt{T}\}}\right)
\\
=&\int_{-\infty}^{\frac{\eta^*-\eta}{\varsigma_{\eta}}\sqrt{T}}\exp \left(-\epsilon_{\eta} \varsigma_{\eta}\sqrt{T}v\right)\left((2\pi)^{-1}\int_{\mathbb{R}}\exp(-iuv)\Psi_{V_T}(u)du\right)dv
\\
=&(2\pi)^{-1}\int_{\mathbb{R}}\Psi_{V_T}(u)\left(\int_{-\infty} ^{\frac{\eta^*-\eta} {\varsigma_{\eta}}\sqrt{T}}\exp\left(-\epsilon_{\eta}\varsigma_{\eta}\sqrt{T}v-iuv\right)dv\right)du
\\
=&-\frac{\exp\left(-\epsilon_{\eta}(\eta^*-\eta)T\right)}{2\pi\epsilon_{\eta} \varsigma_{\eta}\sqrt{T}}\int_{\bR}\left(1+\frac{iu}{\epsilon_{\eta}\varsigma_{\eta}\sqrt{T}}\right)^{-1} \exp\left(-iu\sqrt{T}\frac{\eta^*-\eta}{\varsigma_{\eta}}\right)\Psi_{V_T}(u)du
\\
=&\frac{\exp\left(-\frac{\theta_1^2-\theta_0^2}{2\theta_0} \sqrt{\frac{MT}{2\theta_0}}q_{\alpha}\right)}{2\pi\varsigma_{\eta}\sqrt{T}} \int_{\bR}\left(1-\frac{iu}{\varsigma_{\eta}\sqrt{T}}\right)^{-1} \exp\left(iq_{\alpha}u\right)\Psi_{V_T}(u)du.
\end{align*}
Finally, using \eqref{eq:CharExControl}, and \eqref{eq:CharExpan-VT} and Dominated  Convergent Theorem, we conclude that the last integral converges, as $T\to\infty$, to a finite non-zero constant, and this concludes the proof.

\end{proof}

Now, we are in the position to show that $(R_T^\sharp)$ is not asymptotically the most powerful in $\mathcal{K}^*_{\alpha}$.
First, we note that from \eqref{eq:29} and \eqref{eq:ATVariate} and the definition of $R_T^\sharp$ and $c_{\alpha}^\sharp$, we have
\begin{align*}
c_{\alpha}^\sharp(T)\big/\left(1-\bP^{N,T}_{\theta_1}(R_{T}^\sharp)\right)=&c_{\alpha}^\sharp(T)A_T^{-1}B_T^{-1} \\
=&\exp\left[I(\eta)T+\eta^* T\right]\exp\left[-\mathcal{H}(\epsilon_{\eta})-\mathcal{R}_T(\epsilon_{\eta})\right]B_T^{-1}\\
=&\exp\left(-\frac{\theta_1^2-\theta_0^2}{2\theta_0}\sqrt{\frac{MT}{2\theta_0}}q_{\alpha}\right)B_T^{-1},
\end{align*}
where $\eta$ and $\eta^*$ are given as in Lemma~\ref{lemma:BT-AsymEst}, and then by Lemma~\ref{lemma:BT-AsymEst} we get that
\begin{align*}
c_{\alpha}^\sharp(T)\big/\left(1-\bP^{N,T}_{\theta_1}(R_{T}^\sharp)\right)\sim \sqrt{T},
\end{align*}
which certainly violates condition \eqref{eq:RegT3}. This is an allure that $(R_T^\sharp)$ may not be asymptotically the most powerful test, and
the following two results thoroughly show this.
\begin{theorem}\label{th:MorePowerTests}
The rejection region $(\bar{R}_T)$ defined by
\begin{align*}
\bar{R}_{T}:=\left\{U_T^N: L(\theta_0,\theta_1,U_T^N)\ge \bar{c}_{\alpha}(T)\right\},
\end{align*}
with
\begin{align*}
\bar{c}_{\alpha}(T)=\exp\left(-\frac{(\theta_1-\theta_0)^2}{4\theta_0}MT-\frac{\theta_1^2-\theta_0^2} {2\theta_0}\sqrt{\frac{MT}{2\theta_0}}q_{\alpha}+\bar{\beta}(T)\right),
\end{align*}
and $\bar{\beta}(T)$ satisfying
\begin{align}\label{eq:betaCond}
\bar{\beta}(T)=o(\sqrt{T}),\qquad\limsup_{T\to\infty}\bar{\beta}(T)<0,
\end{align}
is in $\mathcal{K}_{\alpha}^*$, and it is asymptotically more powerful than $(R_T^\sharp)$, that is,
\begin{align*}
\limsup_{T\to\infty}\frac{1-\bP^{N,T}_{\theta_1}(\bar{R}_T)}{1-\bP^{N,T}_{\theta_1}(R_{T}^\sharp)}<1.
\end{align*}
\end{theorem}
\begin{proof}
Following similar arguments as in Proposition~\ref{prop:RTstar-in-K} and using the property that $\bar{\beta}(T)=o(\sqrt{T})$, one can show that $(\bar{R}_{T})\in\cK_\alpha^*$.
 Also, taking
\begin{align*}
\eta^*=-\frac{(\theta_1-\theta_0)^2}{4\theta_0}M-\frac{\theta_1^2-\theta_0^2}{2\theta_0} \sqrt{\frac{M}{2\theta_0T}}q_{\alpha}+\frac{\bar{\beta}(T)}{T},
\end{align*}
and using the same method as in Lemma~\ref{lemma:BT-AsymEst}, we can prove  that
\begin{align*}
B_T\sim \exp\left(-\frac{\theta_1^2-\theta_0^2}{2\theta_0}\sqrt{\frac{MT}{2\theta_0}}q_{\alpha}+\bar{\beta}(T)\right)/\sqrt{T},\quad\textrm{ as }T\to\infty,
\end{align*}
and that
\begin{align*}
\bar{c}_{\alpha}(T)\big/\left(1-\bP^{N,T}_{\theta_1}(\bar{R}_{T})\right)\sim \sqrt{T}.
\end{align*}
Moreover, from the expression of $B_T$ in Lemma~\ref{lemma:BT-AsymEst}, we know that
\begin{align*}
\lim_{T\to\infty}\frac{c_{\alpha}^\sharp(T)\big/\left(1-\bP^{N,T}_{\theta_1}(R_{T}^\sharp)\right)} {\bar{c}_{\alpha}(T)\big/\left(1-\bP^{N,T}_{\theta_1}(\bar{R}_{T})\right)}=1.
\end{align*}
Therefore, by the above limit and \eqref{eq:betaCond},
\begin{align*}
\limsup_{T\to\infty}\frac{1-\bP^{N,T}_{\theta_1}(\bar{R}_T)}{1-\bP^{N,T}_{\theta_1}(R_{T}^\sharp)} =\limsup_{T\to\infty}\frac{\bar{c}_{\alpha}(T)}{c_{\alpha}^\sharp(T)}=\limsup_{T\to\infty}e^{\bar{\beta}(T)}<1.
\end{align*}
\end{proof}

We conclude this section with the proof of Theorem~\ref{Th:NoMostAsymKstar}.
\begin{proof}[Proof of Theorem~\ref{Th:NoMostAsymKstar}]
Since $c_{\alpha}(T)>0$, we may write
\begin{align*}
c_{\alpha}(T)=\exp\left(-\frac{(\theta_1-\theta_0)^2}{4\theta_0}MT-\frac{\theta_1^2-\theta_0^2} {2\theta_0}\sqrt{\frac{MT}{2\theta_0}}q_{\alpha}+\beta_c(T)\right),
\end{align*}
for some  $\beta_c(T)\in\mathbb{R}$.

First, we claim that if $(R_T)$ is asymptotically the most powerful in $\mathcal{K}_{\alpha}^*$, then $\beta_c(T)$ must satisfy condition \eqref{eq:betaCond}.
Indeed, assume that  $\beta_c(T)$ does not satisfy \eqref{eq:betaCond}. Then, first let us assume that $\beta_c(T)=o(\sqrt{T})$ is not satisfied.
Hence, at least one of the followings holds true
\begin{align*}
\liminf_{T\to\infty}\beta_c(T)/\sqrt{T}<0,\qquad\limsup_{T\to\infty}\beta_c(T)/\sqrt{T}>0.
\end{align*}
If the first inequality holds true, then there exists a constant $C>0$, such that for any large $T'$, we can pick up $T>T'$ such that $\beta_c(T)/\sqrt{T}\le-C$.
By \eqref{15} we know that
\begin{align*}
\bP^{N,T}_{\theta_0}(R_T)\ge&\bP^{N,T}_{\theta_0}\left(\sqrt{\frac{2\theta_0}{\sigma^2M}}Y_T\le q_{\alpha}-\frac{\theta_1^2-\theta_0^2} {2\theta_0}\sqrt{\frac{2\theta_0}{MT}}\beta_c(T)\right) \\
\ge&\bP^{N,T}_{\theta_0}\left(\sqrt{\frac{2\theta_0}{\sigma^2M}}Y_T\le q_{\alpha}+\frac{\theta_1^2-\theta_0^2} {2\theta_0}\sqrt{\frac{2\theta_0}{M}}C\right).
\end{align*}
This implies that $\limsup_{T\to\infty}\bP^{N,T}_{\theta_0}(R_T)\ge\alpha+c$, for some constant $c>0$. Hence $(R_T)\notin\mathcal{K}_{\alpha}^*$.
If the second inequality holds true, then there exists a constant $C>0$ and a sequence $\{T_n\}$ which goes to infinity, such that $\beta_c(T_n)/\sqrt{T_n}\ge C$.
Then, it is clear by the definition that $R_{T_n}\subset R_{T_n}^\sharp$, so
\begin{align*}
\frac{1-\bP^{N,T_n}_{\theta_1}(R_{T_n}^\sharp)}{1-\bP^{N,T_n}_{\theta_1}(R_{T_n})}\le1.
\end{align*}
Consequently, by using similar method as in Theorem~\ref{th:MorePowerTests} and the above inequality, we deduce that
\begin{align*}
\liminf_{T\to\infty}\frac{1-\bP^{N,T}_{\theta_1}(\bar{R}_T)}{1-\bP^{N,T}_{\theta_1}(R_T)} \le\liminf_{n\to\infty}\frac{1-\bP^{N,T_n}_{\theta_1}(\bar{R}_{T_n})}{1-\bP^{N,T_n}_{\theta_1}(R_{T_n})}
\le\liminf_{n\to\infty}\frac{1-\bP^{N,T_n}_{\theta_1}(\bar{R}_{T_n})}{1-\bP^{N,T_n}_{\theta_1}(R_{T_n}^\sharp)}<1,
\end{align*}
which contradicts the fact that $(R_T)$ is asymptotically the most powerful in $\mathcal{K}_{\alpha}^*$.

Next, we assume that $\beta_c(T)=o(\sqrt{T})$ is satisfied, but $\limsup_{T\to\infty}\beta_c(T)\ge0$.
Then, by using similar method as in Theorem~\ref{th:MorePowerTests}, we get
\begin{align*}
\liminf_{T\to\infty}\frac{1-\bP^{N,T}_{\theta_1}(\bar{R}_T)}{1-\bP^{N,T}_{\theta_1}(R_{T})} =\liminf_{T\to\infty}\frac{\bar{c}_{\alpha}(T)}{c_{\alpha}(T)}=\liminf_{T\to\infty}e^{\bar{\beta}(T)-\beta_c(T)}<1,
\end{align*}
which again contradicts the original assumption that $(R_T)$ is asymptotically the most powerful in $\mathcal{K}_{\alpha}^*$.
Thus, $\beta_c(T)$ satisfies condition \eqref{eq:betaCond}.

Finally, we assume by contradiction that $(R_T)$ is asymptotically the most powerful in $\mathcal{K}_{\alpha}^*$, then by the above  we know that $\beta_c$ must satisfy condition \eqref{eq:betaCond}. Consider the rejection region $(R_T^+)$ with
\begin{align*}
\beta^+_c(T)=-T^{1/4}|\beta_c(T)|^{1/2}.
\end{align*}
Then it is easy to verify that $\beta^+_c$ still satisfies \eqref{eq:betaCond}, and hence it is in $\mathcal{K}_{\alpha}^*$.
But, in view of Theorem~\ref{th:MorePowerTests}, we have
\begin{align*}
\limsup_{T\to\infty}\frac{1-\bP^{N,T}_{\theta_1}(R^+_T)}{1-\bP^{N,T}_{\theta_1}(R_{T})} =&\limsup_{T\to\infty}\exp\left(-T^{1/4}|\beta_c(T)|^{1/2}-\beta_c(T)\right)
\\
\le&\limsup_{T\to\infty}\exp\left(-T^{1/4}|\beta_c(T)|^{1/2}\left(1-(|\beta_c(T)|/\sqrt{T})^{1/2}\right)\right)=0,
\end{align*}
which gives a contradiction to the fact that $(R_T)$ is asymptotically the most powerful in  $\mathcal{K}_{\alpha}^*$.
This concludes the proof.
\end{proof}

\subsubsection{New class of rejection regions}\label{sec:newClassT}
In order to make $(R_T^\sharp)$ defined by \eqref{eq:RTSharp} to be asymptotically the most powerful, we propose to redefine the class of the tests $\cK_\alpha^*$ by making it slightly smaller. Recall that $(R_T)\in\cK_\alpha^*$, if
$$
\limsup_{T\to\infty}\bP^{N,T}_{\theta_0}(R_T)\le\alpha,
$$
which essentially means that $\bP^{N,T}_{\theta_0}(R_T)$ converges to $\alpha$ (or something smaller). We will look at rejection regions for which $\bP^{N,T}_{\theta_0}(R_T)$ converges to $\alpha$ fast enough, which in a sense is a reasonable assumption.
To write this idea formally, we will need a key lemma which will make clear what is that `reasonable speed of convergence' to be imposed on $\bP^{N,T}_{\theta_0}(R_T)$.

Consider the random process
\begin{align*}
I_T:=-\sqrt{\frac{\theta_0}{2M}}\frac{(\theta_1-\theta_0)\sqrt{T}}{\theta_1+\theta_0}X_T+\sqrt{\frac{2\theta_0}{\sigma^2M}}Y_T,\quad T\geq 0,
\end{align*}
where $X_T$ and $Y_T$ are given by \eqref{eq:heuristT}. Then, we have
\begin{lemma}\label{lemma:ProbExpan-IT}
For any integer $n>0$ and $x,\delta\in\bR$,  the following expansion holds true
\begin{align}\label{eq:ProbExpan-IT}
\bP^{N,T}_{\theta_0}\left(I_T\le x+\delta T^{-1/2}\right)=\sum_{k=0}^nF_k^\delta(x)T^{-k/2}+\mathfrak{R}_{n+1}^{\delta,T}(x)T^{-(n+1)/2},
\end{align}
where $F_k^\delta(\cdot)$ are bounded smooth functions and all their derivatives are bounded, and $\mathfrak{R}_{n+1}^{\delta,T}(\cdot)$ is uniformly bounded in $x$.
\end{lemma}

\begin{proof}
Analogous to \eqref{MomentLLRSplit}, we obtain
\begin{align}\label{eq:MomentLLR-Null}
\bar{\mathcal{L}}_T(\epsilon):=&T^{-1}\ln\bE_{\theta_0}\left[\exp\left(\epsilon\ln L(\theta_0,\theta_1,U_T^N)\right)\right]\notag
\\
&=\bar{\mathcal{L}}(\epsilon)+T^{-1}\bar{\mathcal{H}}(\epsilon)+T^{-1}\bar{\mathcal{R}}_T(\epsilon),
\end{align}
where
\begin{align*}
\bar{\mathcal{L}}(\epsilon)=&\frac{M}{2}\left(\theta_0+(\theta_1-\theta_0)\epsilon-\sqrt{\theta_0^2+(\theta_1^2-\theta_0^2)\epsilon}\right),
\\
\bar{\mathcal{D}}(\epsilon)=&\frac{\theta_0+(\theta_1-\theta_0)\epsilon}{\sqrt{\theta_0^2+(\theta_1^2-\theta_0^2)\epsilon}},
\\
\bar{\mathcal{H}}(\epsilon)=&-\frac{N}{2}\ln\left(\frac{1}{2}+\frac{1}{2}\bar{\mathcal{D}}(\epsilon)\right),
\\
\bar{\mathcal{R}}_T(\epsilon)=&-\frac{1}{2}\sum_{k=1}^N\ln\left(1+\frac{1-\bar{\mathcal{D}}(\epsilon)}
{1+\bar{\mathcal{D}}(\epsilon)}\exp\left(-2\lambda_k^{2\beta}T\sqrt{\theta_0^2+(\theta_1^2-\theta_0^2)\epsilon}\right)\right).
\end{align*}
Then it is not hard to observe that
\begin{align*}
\bar{\mathcal{L}}(\epsilon)=\mathcal{L}(\epsilon-1),\qquad\bar{\mathcal{D}}(\epsilon)=\mathcal{D}(\epsilon-1),\qquad \bar{\mathcal{H}}(\epsilon)=\mathcal{H}(\epsilon-1),\qquad\bar{\mathcal{R}}_T(\epsilon)=\mathcal{R}_T(\epsilon-1).
\end{align*}
Hence,
\begin{align}\label{eq:NAExpTrans}
\bE_{\theta_0}\left[\exp\left(\epsilon\ln L(\theta_0,\theta_1,U_T^N)\right)\right]=\bE_{\theta_1}\left[\exp\left((\epsilon-1)\ln L(\theta_0,\theta_1,U_T^N)\right)\right].
\end{align}
Notice that
\begin{align*}
I_T=-\frac{\sqrt{8\theta_0^3}}{(\theta_1^2-\theta_0^2)\sqrt{TM}}\ln L(\theta_0,\theta_1,U_T^N)-\frac{(\theta_1-\theta_0)\sqrt{\theta_0TM/2}}{\theta_1+\theta_0}.
\end{align*}
We set $\eta=-\frac{(\theta_1-\theta_0)^2}{4\theta_0}M$ in Lemma~\ref{lemma:CharExpan-VT}, and then
\begin{align*}
\epsilon_{\eta}=-1,\qquad\varsigma_{\eta}=(\theta_1^2-\theta_0^2)\sqrt{M/8\theta_0^3}.
\end{align*}
Consequently, using \eqref{eq:NAExpTrans} we can find the characteristic function of $I_T$ under probability measure $\bP^{N,T}_{\theta_0}$,
\begin{align*}
\Psi_{I_T}(u)=&\exp\left(-iu\frac{\theta_1-\theta_0}{\theta_1+\theta_0}\sqrt{\theta_0TM/2}\right)\bE_{\theta_0} \exp\left(-\frac{iu\sqrt{8\theta_0^3}}{(\theta_1^2-\theta_0^2)\sqrt{TM}}\ln L(\theta_0,\theta_1,U_T^N)\right)
\\
=&\exp\left(\frac{iu\eta\sqrt{T}}{\varsigma_{\eta}}\right)\bE_{\theta_0} \exp\left(\frac{iu}{\varsigma_{\eta}\sqrt{T}}\ln L(\theta_0,\theta_1,U_T^N)\right)
\\
=&\exp\left(\frac{iu\eta\sqrt{T}}{\varsigma_{\eta}}\right)\bE_{\theta_1} \exp\left[\left(\epsilon_{\eta}+\frac{iu}{\varsigma_{\eta}\sqrt{T}}\right)\ln L(\theta_0,\theta_1,U_T^N)\right]\\
=&\Psi_{V_T}(-u).
\end{align*}
where the last equality can be verified by direct evaluation.
Therefore, by \eqref{eq:CharExpan-VT}, we get
\begin{align*}
\Psi_{I_T}(u)=e^{-u^2/2}+\sum_{k=1}^{n}\phi_k(u)T^{-k/2}+\mathbf{r}_{n+1}(u),
\end{align*}
where
\begin{align*}
\phi_k(u)=e^{-u^2/2}\sum_{l=k}^{s_1(n-1)}\varphi_{k-1,l}(-u)^{l},\qquad \mathbf{r}_{n+1}(u)=e^{-u^2/2}O\left(\frac{|u|^{s_2(n-1)}+|u|^{s_3(n-1)}}{T^{(n+1)/2}}\right).
\end{align*}
Moreover, the high order term $\mathbf{r}_{n+1}(u) e^{u^2/2}$ is uniformly bounded for $|u|=O(T^{1/6})$.

Now we are ready to prove \eqref{eq:ProbExpan-IT}. First, let us consider the case $\delta=0$, for which we will suppress the index $\delta$ in $F_k^\delta(x)$ and $R_{n+1}^{\delta,T}(x)$. Put
\begin{align}\label{eq:RemainDef}
\mathbf{R}_{n+1}^T(x):=\bP^{N,T}_{\theta_0}\left(I_T\le x\right)-\int_{-\infty}^x\int_{\mathbb{R}}(2\pi)^{-1}e^{-iuv}\left(e^{-u^2/2}+\sum_{k=1}^{n}\phi_k(u)T^{-k/2}\right)dudv.
\end{align}
Note that $\Psi_{I_T}(u)$ also satisfies \eqref{eq:CharExControl}, which admits the representation
\begin{align*}
\int_{\mathbb{R}}e^{iux}d\mathbf{R}_{n+1}^T(x)=&\int_{\mathbb{R}}e^{iux}d\bP^{N,T}_{\theta_0}\left(I_T\le x\right)-e^{-u^2/2}-\sum_{k=1}^{n}\phi_k(u)T^{-k/2}
\\
=&\Psi_{I_T}(u)-e^{-u^2/2}-\sum_{k=1}^{n}\phi_k(u)T^{-k/2}=\mathbf{r}_{n+1}(u).
\end{align*}
Then, following the same approach as \citet[Lemma 4 p.77 and Theorem 12 p.33]{Cramer1970}, we can obtain, for some constant $C>0$,
\begin{align*}
\mathbf{R}_{n+1}^T(x)=O\left(\int_{CT^{1/6}}^{\infty}\frac{|\Psi_{I_T}(u)|}{u}du+T^{-(n+1)/2}\right),
\end{align*}
where the remainder is uniform in $x$. This, combined with \eqref{eq:CharExControl}, yields
\begin{align*}
\mathbf{R}_{n+1}^T(x)=&O\left(O\left(\left(1+\frac{\mu_1 (CT^{1/6})^2}{T}\right)^{-\mu_2 T/2}\right)+T^{-(n+1)/2}\right)
\\
=&O\left(O\left(e^{-\mu_3 T^{1/3}}\right)+T^{-(n+1)/2}\right)=O\left(T^{-(n+1)/2}\right),
\end{align*}
where $\mu_3$ is a positive constant, and the remainder is uniform in $x$. Thus, if we set $\mathfrak{R}^T_{n+1}(x)=T^{(n+1)/2}\mathbf{R}_{n+1}^T(x)$ and recall \eqref{eq:RemainDef}, then \eqref{eq:ProbExpan-IT} for $\delta=0$ follows immediately with
\begin{align*}
F_0(x)=&\int_{-\infty}^x\int_{\mathbb{R}}(2\pi)^{-1}e^{-iuv}e^{-u^2/2}dudv=\Phi(x),\\
F_k(x)=&\int_{-\infty}^x\int_{\mathbb{R}}(2\pi)^{-1}e^{-iuv}\phi_k(u)dudv,\qquad 1\le k\le n.
\end{align*}
Since $\mathbf{R}_{n+1}^T(\cdot)$ is uniform in $x$, we have that $\mathfrak{R}^T_{n+1}(\cdot)$ is uniformly bounded for $T\to\infty$.
Taking into account the form of $\phi_k$, and the definition of $F_k$, the  smoothness and boundedness of $F_k(\cdot)$ and its derivatives follow immediately.

Now we consider the case $\delta\neq0$. Since $F_k(x)$ has bounded derivatives, we have the expansions
\begin{align*}
F_k(x+\delta T^{-1/2})=&\sum_{j=0}^{n-k}\frac{F_k^{(j)}(x)}{j!}\delta^j T^{-j/2}+\frac{F_k^{(n-k+1)}(x_k^\delta)}{(n-k+1)!}\delta^{n-k+1} T^{-(n-k+1)/2}
\\
=&\sum_{j=0}^{n-k}\frac{F_k^{(j)}(x)}{j!}\delta^j T^{-j/2}+O\left(T^{-(n-k+1)/2}\right),\qquad 0\le k\le n,
\end{align*}
where $x_k^\delta$ is some number between $x$ and $\delta T^{-1/2}$, and the remainder is uniformly bounded in $x$. Note that, from the above case $\delta=0$, and with $x:=x+\delta T^{-1/2}$ in  \eqref{eq:ProbExpan-IT}, we already have
\begin{align*}
\bP^{N,T}_{\theta_0}\left(I_T\le x+\delta T^{-1/2}\right)=&\sum_{k=0}^nF_k(x+\delta T^{-1/2})T^{-k/2}+\mathfrak{R}_{n+1}^T(x+\delta T^{-1/2})T^{-(n+1)/2} \\
=&\sum_{k=0}^nF_k(x+\delta T^{-1/2})T^{-k/2}+O\left(T^{-(n+1)/2}\right).
\end{align*}
Combining the above two relations yields
\begin{align*}
\bP^{N,T}_{\theta_0}\left(I_T\le x+\delta T^{-1/2}\right)=&\sum_{k=0}^n\left(\sum_{j=0}^{k}\frac{F_{k-j}^{(j)}(x)}{j!}\delta^j\right)T^{-k/2}+O\left(T^{-(n+1)/2}\right).
\end{align*}
This concludes the proof.
\end{proof}

For a fixed parameter $\delta\in\bR$, we define the following family of rejection regions
\begin{align}\label{def:RTdelta}
R_T^\delta:=\{U_T^N: I_T\le q_{\alpha}+\delta T^{-1/2}\}.
\end{align}
Alternatively, one can show that $R_T^\delta$ can be written as a likelihood ratio test, as given in \eqref{eq:RTDelta}.
The parameter $\delta$ is meant to control the speed of convergence of  $\bP^{N,T}_{\theta_0}\left(R_T^\delta\right)$ to $\alpha$. Indeed,
by Lemma~\ref{lemma:ProbExpan-IT}, with $n=1$ and $x=q_{\alpha}$, we have
\begin{align}\label{eq:Rdelta-NullAsym}
\bP^{N,T}_{\theta_0}\left(R_T^\delta\right)=&\bP^{N,T}_{\theta_0}\left(I_T\le q_{\alpha}+\delta T^{-1/2}\right)\notag \\
=&F_0(q_{\alpha})+\left(F_0'(q_{\alpha})\delta+F_1(q_{\alpha})\right)T^{-1/2}+O(T^{-1})\notag \\
=&\alpha+\left((2\pi)^{-1/2}e^{-q_{\alpha}^2/2}\delta+F_1(q_{\alpha})\right)T^{-1/2}+O(T^{-1}) \notag \\
=&\alpha+\alpha_1(\delta)T^{-1/2}+O(T^{-1}),
\end{align}
where $\alpha_1(\delta) = (2\pi)^{-1/2}e^{-q_{\alpha}^2/2}\delta+F_1(q_{\alpha})$.
Since $\delta$ can be chosen arbitrarily, we may let $\alpha_1(\delta)>0$.
As next result shows, there exists a computable form for the threshold $\alpha_1$, that can be conveniently  used in applications and numerical evaluations.

\begin{proposition}\label{prop:alpha1}
For any $\delta\in\bR$, the following identity holds true
\begin{equation}\label{eq:alpha1explicit}
\alpha_1(\delta) = (2\pi)^{-1/2}e^{-q_{\alpha}^2/2}\delta + \frac{e^{-q_{\alpha}^2/2}}{2\sqrt{\pi M\theta_0}}\left(\frac{(\theta_1-\theta_0)N}{2(\theta_1+\theta_0)}+1-q_{\alpha}^2\right).
\end{equation}
\end{proposition}
\begin{proof}
From \eqref{VCharLogExpan} we know
\begin{align*}
\ln\Psi_{V_T}(u)=-\frac{u^2}{2}+\left(\frac{iu^3}{\sqrt{2\theta_0M}}+\frac{iu(\theta_1-\theta_0)N} {2(\theta_1+\theta_0)\sqrt{2\theta_0M}}\right)T^{-1/2}+O(T^{-1}),
\end{align*}
which implies
\begin{align*}
\Psi_{I_T}(u)=\Psi_{V_T}(-u)=e^{-u^2/2}\left[1-\left(\frac{iu^3}{\sqrt{2\theta_0M}}+\frac{iu(\theta_1-\theta_0)N} {2(\theta_1+\theta_0)\sqrt{2\theta_0M}}\right)T^{-1/2}+O(T^{-1})\right].
\end{align*}
Consequently,
\begin{align*}
\phi_1(u)=-e^{-u^2/2}\left(\frac{iu^3}{\sqrt{2\theta_0M}}+\frac{iu(\theta_1-\theta_0)N} {2(\theta_1+\theta_0)\sqrt{2\theta_0M}}\right).
\end{align*}
Thus,
\begin{align*}
\alpha_1(\delta)=&(2\pi)^{-1/2}e^{-q_{\alpha}^2/2}\delta+F_1(q_{\alpha}) \\
=&(2\pi)^{-1/2}e^{-q_{\alpha}^2/2}\delta+\int_{-\infty}^{q_{\alpha}}\int_{\mathbb{R}}(2\pi)^{-1}e^{-iuv}\phi_1(u)dudv \\
=&(2\pi)^{-1/2}e^{-q_{\alpha}^2/2}\delta + \frac{e^{-q_{\alpha}^2/2}}{2\sqrt{\pi M\theta_0}}\left(\frac{(\theta_1-\theta_0)N}{2(\theta_1+\theta_0)}+1-q_{\alpha}^2\right).
\end{align*}

\end{proof}

Finally, with the key asymptotic \eqref{eq:Rdelta-NullAsym} at hand, that depicts the precise first nontrivial rate of convergence of Type~I error to $\alpha$, we consider the class of test with type one error converging to $\alpha$, with the rate of convergence at least $\alpha_1T^{-1/2}$. Namely, we define the following class of tests (also defined in Section~\ref{sec:mainResultsTime}, but for convenience we repeat it here too)
\begin{align}\label{def:NewAsymClass}
\mathcal{K}_{\alpha}^\sharp(\delta)=\left\{(R_T): \limsup_{T\to\infty}\left(\bP^{N,T}_{\theta_0}(R_T)-\alpha\right)\sqrt{T}\le\alpha_1(\delta)\right\}.
\end{align}

Now we ready to prove the main result for the case of large time asymptotics: for any $\delta\in\bR$, the rejection region $(R_T^\delta)$ is asymptotically the most powerful in the class $\mathcal{K}^\sharp_{\alpha}(\delta)$.

\begin{proof}[Proof of Theorem~\ref{th:MainResult1}]
The fact that $(R_T^\delta)\in\mathcal{K}^\sharp_{\alpha}(\delta)$ follows immediately from  \eqref{eq:Rdelta-NullAsym}.

Using the representation \eqref{eq:RTDelta} of $R_T^\delta$, and following the lines of the proof of Theorem~\ref{th:NPLemma}, for any $(R_T)\in\mathcal{K}^\sharp_{\alpha}$, we have
\begin{align*}
\frac{1-\bP^{N,T}_{\theta_1}(R_T)}{1-\bP^{N,T}_{\theta_1}(R_T^\delta)} \ge1+\frac{c^\delta_{\alpha}(T)}{1-\bP^{N,T}_{\theta_1}(R_T^\delta)}\left(\bP^{N,T}_{\theta_0}(R_T^\delta)-\bP^{N,T}_{\theta_0}(R_T)\right).
\end{align*}
Taking here the liminf of both sides, we deduce
\begin{align*}
\underset{T\to\infty}{\liminf}\frac{1-\bP^{N,T}_{\theta_1}(R_T)}{1-\bP^{N,T}_{\theta_1}(R_T^\delta)} \ge&1+\liminf_{T\to\infty}\frac{c^\delta_{\alpha}(T)}{1-\bP^{N,T}_{\theta_1}(R_T^\delta)}\left(\bP^{N,T}_{\theta_0}(R_T^\delta)-\alpha\right)\\
&\quad-\limsup_{T\to\infty}\frac{c^\delta_{\alpha}(T)}{1-\bP^{N,T}_{\theta_1}(R_T^\delta)}\left(\bP^{N,T}_{\theta_0}(R_T)-\alpha\right).
\end{align*}
By the same method used in Lemma \ref{th:MorePowerTests}, we can still obtain
\begin{align*}
c^\delta_{\alpha}(T)/\left(1-\bP^{N,T}_{\theta_1}(R_{T}^\delta)\right)\sim \sqrt{T},
\end{align*}
which permits us to set $C=\displaystyle\lim_{T\to\infty}\frac{c^\delta_{\alpha}(T)}{\left(1-\bP^{N,T}_{\theta_1}(R_{T}^\delta)\right)\sqrt{T}}$. Then we may deduce
\begin{align*}
\liminf_{T\to\infty}\frac{c^\delta_{\alpha}(T)}{1-\bP^{N,T}_{\theta_1}(R_T^\delta)}\left(\bP^{N,T}_{\theta_0}(R_T^\delta)-\alpha\right)&= C\lim_{T\to\infty}\left(\bP^{N,T}_{\theta_0}(R_T^\delta)-\alpha\right)\sqrt{T}=C\alpha_1,
\\
\limsup_{T\to\infty}\frac{c^\delta_{\alpha}(T)}{1-\bP^{N,T}_{\theta_1}(R_T^\delta)}\left(\bP^{N,T}_{\theta_0}(R_T)-\alpha\right)&\le C\limsup_{T\to\infty}\left(\bP^{N,T}_{\theta_0}(R_T)-\alpha\right)\sqrt{T}\le C\alpha_1.
\end{align*}
To sum up all the above, we obtain
\begin{align*}
\liminf_{T\to\infty}\frac{1-\bP^{N,T}_{\theta_1}(R_T)}{1-\bP^{N,T}_{\theta_1}(R_T^\delta)}\ge1.
\end{align*}
The proof is finished.
\end{proof}

\begin{remark}
Note that in particular, for $\delta=0$,  we have that
$R_T^\sharp=R^0_T=\{U_T^N: I_T\le q_{\alpha}\}$,
although $\cK_\alpha^*\neq \cK_\alpha^\sharp(0)$.
\end{remark}

\subsection{Large number of Fourier modes: proofs}\label{sec:AsympMethodN}
In this section we will take a similar asymptotic approach but with number of Fourier coefficients $N\to\infty$ , and for a fixed time horizon $T$.
While the general ideas are parallel to those from Section~\ref{ssec:AsymptoticT}, the corresponding technical results are proved essentially from scratch, using different technics.
The main challenge is due to the fact that there are no existing results on `large deviations' for $N\to\infty$.
However, we want to stress out that those general ideas from large time asymptotic regime gave us the important insights on how to study the case $N\to\infty$.

As before, we are interested in finding a class of rejection regions such that the rejection region of the form
\begin{align}\label{def:RNStar}
\widetilde{R}_N=\left\{U_T^N: L(\theta_0,\theta_1,U_T^N)\ge \widetilde{c}_{\alpha}(N)\right\},
\end{align}
where $\widetilde{c}_{\alpha}(N)$ is a positive sequence indexed by $N$, is asymptotically the most powerful in that class. First, let us derive heuristically the form of $\widetilde{c}_\alpha(N)$.
Clearly, we are looking for tests such that $\bP_{\theta_0}^{N,T}(\widetilde{R}_N)\to\alpha$.
By similar argumentations to \eqref{eq:heuristT},  we have
\begin{align}\label{eq:heuristN}
\bP^{N,T}_{\theta_0}&(L(\theta_0,\theta_1,U_T^N)\ge \widetilde{c}_{\alpha}(N))\notag \\
=&\bP^{N,T}_{\theta_0}\left(\frac{\sqrt{\theta_0}(\theta_1-\theta_0)}{\sigma^2\sqrt{2TM}(\theta_1+\theta_0) }\left(\sum_{k=1}^N\lambda_k^{2\beta+2\gamma}u_k^2(T)-\frac{\sigma^2}{2\theta_0}N\right)-\frac{\sqrt{2\theta_0}}{\sigma\sqrt{TM} }\sum_{k=1}^N\lambda_k^{2\beta+\gamma}\int_0^Tu_kdw_k\right.\notag \\
&\qquad\qquad\qquad\qquad\left.\ge \frac{\sqrt{8\theta_0^3}\ln \widetilde{c}_{\alpha}(N)}{\sqrt{TM}(\theta_1^2-\theta_0^2)}+\frac{\sqrt{2\theta_0TM}(\theta_1-\theta_0) }{2(\theta_1+\theta_0)}-\frac{(\theta_1-\theta_0)N}{\sqrt{8\theta_0TM}(\theta_1+\theta_0)}\right)\notag \\
=&\bP^{N,T}_{\theta_0}\left(X^N-Y^N\ge \frac{\sqrt{8\theta_0^3}\ln \widetilde{c}_{\alpha}(N)}{\sqrt{TM}(\theta_1^2-\theta_0^2)}+\frac{\sqrt{2\theta_0TM}(\theta_1-\theta_0) }{2(\theta_1+\theta_0)}-\frac{(\theta_1-\theta_0)N}{\sqrt{8\theta_0TM}(\theta_1+\theta_0)}\right),
\end{align}
where
\begin{align*}
X^N:=&\frac{\sqrt{\theta_0}(\theta_1-\theta_0)}{\sigma^2\sqrt{2TM}(\theta_1+\theta_0) }\left(\sum_{k=1}^N\lambda_k^{2\beta+2\gamma}u_k^2(T)-\frac{\sigma^2}{2\theta_0}N\right),
\\
Y^N:=&\frac{\sqrt{2\theta_0}}{\sigma\sqrt{TM} }\sum_{k=1}^N\lambda_k^{2\beta+\gamma}\int_0^Tu_kdw_k.
\end{align*}
Notice that, the quantity $M=\sum_{k=1}^N\lambda_k^{2\beta}$ depends on our main variable $N$, in contrast to Section~\ref{ssec:AsymptoticT} where it was just a constant independent of $T$.
From the above, for $\delta\in\bR$, we have
\begin{align}\label{eq:ProbSplitN1}
\bP^{N,T}_{\theta_0}&(L(\theta_0,\theta_1,U_T^N)\ge \widetilde{c}_{\alpha}(N))+\bP^{N,T}_{\theta_0}\left(-X^N\ge\delta\right)\notag \\
\ge&\bP^{N,T}_{\theta_0}\left(-Y^N\ge \frac{\sqrt{8\theta_0^3}\ln \widetilde{c}_{\alpha}(N)}{\sqrt{TM}(\theta_1^2-\theta_0^2)}+\frac{\sqrt{2\theta_0TM}(\theta_1-\theta_0) }{2(\theta_1+\theta_0)}-\frac{(\theta_1-\theta_0)N}{\sqrt{8\theta_0TM}(\theta_1+\theta_0)}+\delta\right),
\end{align}
and
\begin{align}\label{eq:ProbSplitN2}
&\bP^{N,T}_{\theta_0}(L(\theta_0,\theta_1,U_T^N)\ge \widetilde{c}_{\alpha}(N))\le\bP^{N,T}_{\theta_0}\left(X^N\ge\delta\right)\notag
\\
&+\bP^{N,T}_{\theta_0}\left(-Y^N\ge \frac{\sqrt{8\theta_0^3}\ln \widetilde{c}_{\alpha}(N)}{\sqrt{TM}(\theta_1^2-\theta_0^2)}+\frac{\sqrt{2\theta_0TM}(\theta_1-\theta_0) }{2(\theta_1+\theta_0)}-\frac{(\theta_1-\theta_0)N}{\sqrt{8\theta_0TM}(\theta_1+\theta_0)}-\delta\right).
\end{align}
By Lemma~\ref{lemma:YN}(i) we have that  $Y^N\overset{d}{\rightarrow}\mathcal{N}(0,1)$ as $N\to\infty$. Using Lemma~\ref{lemma:YN}(ii), and by taking $\delta=N^{-\beta/2d}$, we deduce
\begin{align*}
\bP^{N,T}_{\theta_0}\left(X^N\ge\delta\right)\to0,\quad\textrm{ as $N\to\infty$.}
\end{align*}
Thus, it is reasonable to choose $\widetilde{c}_{\alpha}(N)$ such that
\begin{align*}
\frac{\sqrt{8\theta_0^3}\ln \widetilde{c}_{\alpha}(N)}{\sqrt{TM}(\theta_1^2-\theta_0^2)}+\frac{\sqrt{2\theta_0TM}(\theta_1-\theta_0) }{2(\theta_1+\theta_0)}-\frac{(\theta_1-\theta_0)N}{\sqrt{8\theta_0TM}(\theta_1+\theta_0)}=-q_{\alpha}.
\end{align*}
Hence, we take
\begin{align}\label{eq:cAlphahat}
\widehat{c}_{\alpha}(N)=\exp\left(-\frac{(\theta_1-\theta_0)^2 TM}{4\theta_0}+\frac{(\theta_1-\theta_0)^2N}{8\theta_0^2 }-\frac{\sqrt{TM}(\theta_1^2-\theta_0^2)}{\sqrt{8\theta_0^3}}q_{\alpha}\right).
\end{align}
We denote by $(\widehat{R}_N)$  the rejection region given by \eqref{def:RNStar} with $\widehat{c}_{\alpha}$ as in \eqref{eq:cAlphahat}.
Similarly as in $T$ part, we will see that $(\widehat{R}_N)$ is not asymptotically the most powerful in $\widetilde{\mathcal{K}}_{\alpha}$.
However, first we will present some results on sharp large deviation bounds for large $N$.

\subsubsection{Sharp large deviation principle, large $N$}

Akin to \eqref{eq:notationLHR} and \eqref{MomentLLRSplit}, we define
\begin{align*}
\mathcal{L}_N(\epsilon):=&M^{-1}\ln\bE_{\theta_1}\left[\exp\left(\epsilon\ln L(\theta_0,\theta_1,U_T^N)\right)\right] =\widetilde{\mathcal{L}}(\epsilon)+NM^{-1}\widetilde{\mathcal{H}}(\epsilon)+M^{-1}\widetilde{\mathcal{R}}_N(\epsilon),
\end{align*}
where
\begin{align*}
\widetilde{\mathcal{L}}(\epsilon)=&\frac{T}{2}\left(\theta_1+(\theta_1-\theta_0)\epsilon -\sqrt{\theta_1^2+(\theta_1^2-\theta_0^2)\epsilon}\right),
\\
\widetilde{\mathcal{H}}(\epsilon)=&-\frac{1}{2}\ln\left(\frac{1}{2}+\frac{1}{2}\mathcal{D}(\epsilon)\right),
\\
\widetilde{\mathcal{R}}_N(\epsilon)=&-\frac{1}{2}\sum_{k=1}^N\ln\left(1+\frac{1-\mathcal{D}(\epsilon)}{1+\mathcal{D}(\epsilon)} \exp\left(-2\lambda_k^{2\beta}T\sqrt{\theta_1^2+(\theta_1^2-\theta_0^2)\epsilon}\right)\right).
\end{align*}
We also define
\begin{align}\label{62}
\widetilde{A}_N=&\exp\left[M(\mathcal{L}_N(\widetilde{\epsilon}_{\eta})-\eta\widetilde{\epsilon}_{\eta})\right],\notag
\\
\widetilde{B}_N=&\bE_N\left(\exp\left[-\widetilde{\epsilon}_{\eta}(\ln L(\theta_0,\theta_1,U_T^N)-\eta M)\right]\1_{\{\ln L(\theta_0,\theta_1,U_T^N)\le\eta^* M\}}\right),
\end{align}
where $\eta$ and $\eta^*$ are some numbers which may depend on $N$, and $\bE_N$ is the expectation after the change of probability measure
\begin{align}\label{eq:MeasureChange2}
\frac{d\bQ_N}{d\bP^{N,T}_{\theta_1}}=\exp\left(\widetilde{\epsilon}_{\eta}\ln L(\theta_0,\theta_1,U_T^N)-M\mathcal{L}_N(\widetilde{\epsilon}_{\eta})\right).
\end{align}
Clearly,
\begin{align*}
\bP^{N,T}_{\theta_1}\left(\ln L(\theta_0,\theta_1,U_T^N)\le\eta^* M\right)=\widetilde{A}_N\widetilde{B}_N.
\end{align*}
Naturally, by taking $\widetilde{\epsilon}_{\eta}$ such that $\widetilde{\mathcal{L}}'(\widetilde{\epsilon}_{\eta})=\eta$, we get
\begin{align}\label{eq:epsilonT}
\widetilde{\epsilon}_{\eta}=\frac{(\theta_1^2-\theta_0^2)^2T^2-4\theta_1^2(-2\eta+(\theta_1-\theta_0)T)^2}{4(\theta_1^2-\theta_0^2)(-2\eta+(\theta_1-\theta_0)T)^2},
\end{align}
and then by direct evaluations we find
\begin{align}\label{64}
\widetilde{A}_N=&\exp\left[M(\mathcal{L}(\widetilde{\epsilon}_{\eta})-\eta\widetilde{\epsilon}_{\eta})\right]\exp\left[N\widetilde{\mathcal{H}} (\widetilde{\epsilon}_{\eta})+\widetilde{\mathcal{R}}_N(\widetilde{\epsilon}_{\eta})\right]\notag
\\
=&\exp\left(-\widetilde{I}(\eta)M\right)\left(\frac{1}{2}+\frac{1}{2}\mathcal{D}(\widetilde{\epsilon}_{\eta})\right)^{-N/2}\notag
\\
&\qquad\qquad\prod_{k=1}^N\left(1+\frac{1-\mathcal{D}(\widetilde{\epsilon}_{\eta})}{1+\mathcal{D}(\widetilde{\epsilon}_{\eta})} \exp\left(-2\lambda_k^{2\beta}T\sqrt{\theta_1^2+(\theta_1^2-\theta_0^2)\widetilde{\epsilon}_{\eta}}\right)\right)^{-1/2},
\end{align}
where
\begin{align*}
\widetilde{I}(\eta)=-\frac{(4\theta_1\eta-(\theta_1-\theta_0)^2T)^2}{8(2\eta-(\theta_1-\theta_0)T)(\theta_1^2-\theta_0^2)}.
\end{align*}
Next we present three technical lemmas with their proofs deferred to the Appendix~\ref{TechLemma}.

\begin{lemma}\label{lemma:CharExpan-VN}
Under probability measure $\bQ_N$, the following holds true
\begin{align*}
\widetilde{V}_N:=\frac{\ln L(\theta_0,\theta_1,U_T^N)-\eta M-\widetilde{\mathcal{H}}'(\widetilde{\epsilon}_{\eta})N}{\zeta_{\eta}\sqrt{M}}\overset{d}{\longrightarrow}\mathcal{N}(0,1), \quad\textrm{ as $N\to\infty$},
\end{align*}
where $\eta<\frac{(\theta_1-\theta_0)^2 T}{4\theta_1}$, and
\begin{align*}
\zeta_{\eta}^2=\widetilde{\mathcal{L}}''(\widetilde{\epsilon}_{\eta})=\frac{(-2\eta+(\theta_1-\theta_0)T)^3}{(\theta_1^2-\theta_0^2)T^2}.
\end{align*}
More precisely, for large enough $N$, the characteristic function of $\widetilde{V}_N$ under measure $\bQ_N$, denoted by $\widetilde{\Psi}_N$, has the following expansion
\begin{align}\label{eq:CharExpan-VN}
\widetilde{\Psi}_N(u)=\exp\left(-\frac{u^2}{2}\right) &\left[1-\frac{i\widetilde{\mathcal{L}}^{(3)}(\widetilde{\epsilon}_{\eta})u^3}{6\zeta_{\eta}^3\sqrt{M}} -\frac{\widetilde{\mathcal{H}}''(\widetilde{\epsilon}_{\eta})Nu^2}{2\zeta_{\eta}^2M} +\frac{C_\eta u}{4\pi\zeta_{\eta}\sqrt{M}}\right.\notag
\\
&\qquad\quad\qquad\left.+O\left(\frac{|u|+|u|^6}{M}+\frac{|u|^3N}{M^{3/2}}+\frac{u^4N^2}{M^2}\right)\right],
\end{align}
where $C_\eta$ is some constant depending only on $\eta$. Moreover, the remainder is uniformly bounded in $u$ as long as $|u|=O(N^{\nu})$, $\nu=\min\{1/6+\beta/(3d),\beta/d\}$.
\end{lemma}

\begin{lemma}\label{lemma:BN-AsymEst}
For
\begin{align*}
\eta=-\frac{(\theta_1-\theta_0)^2}{4\theta_0}T,\qquad \eta^*=-\frac{(\theta_1-\theta_0)^2 T}{4\theta_0}+\frac{(\theta_1-\theta_0)^2N}{8\theta_0^2M} -\frac{\sqrt{T}(\theta_1^2-\theta_0^2)}{\sqrt{8\theta_0^3M}}q_{\alpha} -\frac{\sqrt{T}(\theta_1^2-\theta_0^2)}{\sqrt{8\theta_0^3}M}\delta,
\end{align*}
as $N\to\infty$, the following asymptotic holds true
\begin{align*}
\widetilde{B}_N\sim \exp\left[(\eta^*-\eta)M\right]/\sqrt{M}.
\end{align*}
\end{lemma}

\begin{lemma}\label{lemma:ProbExpan-IN}
For any $x,\delta\in\bR$, we have the following expansion,
\begin{align}\label{eq:ProbExpan-IN}
\bP^{N,T}_{\theta_0}\left(Y^N-X^N\le x+\delta M^{-1/2}\right)=&\Phi(x)+\Phi_1^\delta(x)M^{-1/2}+\Phi_2^\delta(x)NM^{-1}\notag
\\
&+\mathfrak{R}_N^{\delta}(x)\left(M^{-1}+NM^{-3/2}+N^2M^{-2}\right),
\end{align}
where $\Phi(\cdot)$ is the distribution of a standard Gaussian random variable, and
\begin{align*}
\Phi_1^\delta(x)=&\left[\frac{\theta_1-\theta_0}{4\sqrt{\pi\theta_0T}(\theta_1+\theta_0)} \left(\sum_{k=1}^{\infty} e^{-2\theta_0T\lambda_k^{2\beta}}\right)+\frac{1}{2\sqrt{\pi\theta_0T}}(1-x^2)+(2\pi)^{-1/2}\delta\right]e^{-x^2/2},
\\
\Phi_2^\delta(x)=&\frac{(\theta_1-\theta_0)(5\theta_1^2+6\theta_1\theta_0-3\theta_0^2)} {8\sqrt{2\pi}\theta_0(\theta_1+\theta_0)(\theta_1^2-\theta_0^2)T}xe^{-x^2/2},
\end{align*}
and $\mathfrak{R}_N^{\delta}(\cdot)$ is uniformly bounded in $x$.
\end{lemma}

\subsubsection{New class of rejection regions}
Relation \eqref{eq:ProbExpan-IN} is the counterpart of \eqref{eq:Rdelta-NullAsym}, and by similarity we take the `refined' class of tests\footnote{The class of tests $\widehat{\mathcal{K}}_{\alpha}(\delta)$ was already defined in Section~\ref{sec:mainResultsSpace}, relation \eqref{def:Kbar}, but for convenience we repeat it here too.} $\widehat{\mathcal{K}}_{\alpha}(\delta)$ with Type~I error converging to $\alpha$ at rate at least $\widehat\alpha_1(\delta)M^{-1/2}$,
\begin{align}\label{def:Kbar-1}
\widehat{\mathcal{K}}_{\alpha}(\delta):=\left\{(R_N): \limsup_{N\to\infty}\left(\bP^{N,T}_{\theta_0}(R_N)-\alpha\right)\sqrt{M}\le\widehat{\alpha}_1(\delta)\right\},
\end{align}
where $\delta\in\bR$ is fixed, and
\begin{align}\label{def:tildeAlpha1-1}
\widehat{\alpha}_1(\delta):=\left\{\begin{array}{ccc}\Phi_1^\delta(q_\alpha),&\textrm{ if $\beta/d>1/2$} \\ \Phi_1^\delta(q_\alpha)+\sqrt{\frac{2\beta/d+1}{\varpi^\beta}}\Phi_2^\delta(q_\alpha),&\textrm{ if $\beta/d=1/2$}\end{array}\right.,
\end{align}
where $\Phi_j^\delta(\cdot), \ j=1,2$, are given by \eqref{eq:ProbExpan-IN}.

Note that the rate of convergence depends on the fractional power $\beta$ of the Laplace operator and the dimension $d$ of the space.
The reasons is that for $\beta/d>1/2$, $M^{-1/2}$ is the lowest order term, while for $\beta/d=1/2$, $M^{-1/2}$ and $NM^{-1}$ together make the lowest order term, which leads to different values of $\widehat{\alpha}_1(\delta)$.

Now we are ready to prove the main result for the case of large number of Fourier modes.

\begin{proof}[Proof of Theorem~\ref{th:MainResult2}]
Assume that $\beta/d\ge1/2$, and take a fixed $\delta\in\bR$.
We will prove that the test $\widehat{R}_N^\delta$, defined by  \eqref{eq:bestR_N}, is asymptotically the most powerful in $\widehat{\mathcal{K}}_{\alpha}(\delta)$.

It is straightforward  to show that
\begin{align}\label{def:RNdelta}
\widehat{R}_N^\delta=\left\{U_T^N: Y^N-X^N\le q_{\alpha}+\delta M^{-1/2}\right\}.
\end{align}

The fact that $(\widehat{R}_N^\delta)\in\widehat{\mathcal{K}}_{\alpha}(\delta)$ follows immediately from  \eqref{eq:ProbExpan-IN}.

From \eqref{62} and \eqref{64} we know that
\begin{align*}
\widehat{c}^\delta_{\alpha}(N)\big/\left(1-\bP^{N,T}_{\theta_1}(\widehat{R}_N^\delta)\right)=&\widehat{c}^ \delta_{\alpha}(N)\widetilde{A}_N^{-1}\widetilde{B}_N^{-1}
\\
=&\exp\left[\widetilde{I}(\eta)M+\eta^*M\right]\exp\left[N\widetilde{\mathcal{H}} (\widetilde{\epsilon}_{\eta})+\widetilde{\mathcal{R}}_N(\widetilde{\epsilon}_{\eta})\right]\widetilde{B}_N^{-1}
\\
=&\exp\left(\frac{(\theta_1-\theta_0)^2}{8\theta_0^2}N-\frac{\sqrt{T}(\theta_1^2-\theta_0^2) q_{\alpha}}{\sqrt{8\theta_0^3}}\sqrt{M}-\frac{\sqrt{T}(\theta_1^2-\theta_0^2)}{\sqrt{8\theta_0^3}}\delta\right)\widetilde{B}_N^{-1},
\end{align*}
where $\eta$ and $\eta^*$ are given in Lemma~\ref{lemma:BN-AsymEst}. Then, by Lemma~\ref{lemma:BN-AsymEst}, we can set \begin{align*}
C=\displaystyle\lim_{N\to\infty}\frac{\widehat{c}^\delta_{\alpha}(N)}{\left(1-\bP^{N,T}_{\theta_1}(\widehat{R}_N^\delta)\right)\sqrt{M}}.
\end{align*}
Following the same lines as in the proof of Theorem~\ref{th:MainResult1}, we deduce
\begin{align*}
\underset{N\to\infty}{\liminf} \frac{1-\bP^{N,T}_{\theta_1}(R_N)}{1-\bP^{N,T}_{\theta_1}(\widehat{R}_N^\delta)}\ge1,\qquad \textrm{ for all } (R_N)\in\widehat{\mathcal{K}}_{\alpha}(\delta).
\end{align*}
This ends the proof.
\end{proof}

\subsubsection{On non-optimality of class $\widetilde{\mathcal{K}}_{\alpha}$ }

In this section we will show that indeed there is no asymptotically most powerful test of likelihood ratio type in the class $\widetilde{\mathcal{K}}_{\alpha}$.
Parallel to Theorem~\ref{th:regRegAT}, we have the following result.
\begin{theorem}\label{th:regRegAN}
Consider the rejection region given by \eqref{def:RNStar}. If
\begin{align}
&\lim_{N\to\infty}\bP^{N,T}_{\theta_0}(\widetilde{R}_N)=\alpha, \label{eq:RegN2}
\\
&\lim_{N\to\infty}\frac{\widetilde{c}_{\alpha}(N)}{1-\bP^{N,T}_{\theta_1}(\widetilde{R}_N)}<\infty, \label{eq:RegN3}
\end{align}
then $(\widetilde{R}_N)$ is asymptotically the most powerful in $\widetilde{\mathcal{K}}_{\alpha}$.
\end{theorem}
\noindent The proof is parallel to that of Theorem \ref{th:regRegAT}, and we omit it here.

Taking $\delta=0$ in Lemma~\ref{lemma:BN-AsymEst} and the proof of Theorem~\ref{th:MainResult2}, we can show that
\begin{align*}
\widehat{c}_{\alpha}(N)\big/\left(1-\bP^{N,T}_{\theta_1}(\widehat{R}_N)\right)\sim\sqrt{M}\sim N^{\beta/d+1/2},
\end{align*}
which certainly violates condition \eqref{eq:RegN3}, and allures that $(\widehat{R}_N)$ may not be asymptotically the most powerful test.
The following result thoroughly shows this.

\begin{theorem}\label{th:MorePowerTests-N}
The rejection region $(\bar{R}_N)$ given by \eqref{def:RNStar} with
\begin{align}\label{28}
\bar{c}_{\alpha}(N)=\exp\left(-\frac{(\theta_1-\theta_0)^2 TM}{4\theta_0}+\frac{(\theta_1-\theta_0)^2N}{8\theta_0^2 }-\frac{\sqrt{TM}(\theta_1^2-\theta_0^2)}{\sqrt{8\theta_0^3}}q_{\alpha}+\bar{\beta}(N)\right),
\end{align}
and $\bar{\beta}(N)$ satisfying
\begin{align}\label{68}
\bar{\beta}(N)=o(N^{\beta/d+1/2}),\qquad\limsup_{N\to\infty}\bar{\beta}(N)<0,
\end{align}
is in $\widetilde{\mathcal{K}}_{\alpha}$ and it is asymptotically more powerful than $(\widehat{R}_N)$, that is,
\begin{align*}
\limsup_{N\to\infty}\frac{1-\bP^{N,T}_{\theta_1}(\bar{R}_N)}{1-\bP^{N,T}_{\theta_1}(\widehat{R}_N)}<1.
\end{align*}
\end{theorem}
\begin{proof}
Following similar lines of the proof as in Lemma~\ref{lemma:ProbExpan-IN} with $x=q_\alpha$, $\delta=\bar{\beta}(N)$, and using the property that $\bar{\beta}(N)=o(N^{\beta/d+1/2})$, one can show that $(\bar{R}_N)\in\widetilde{\cK}_\alpha$. Also, taking
\begin{align*}
\eta^*=-\frac{(\theta_1-\theta_0)^2 T}{4\theta_0}+\frac{(\theta_1-\theta_0)^2N}{8\theta_0^2M }-\frac{\sqrt{T}(\theta_1^2-\theta_0^2)}{\sqrt{8\theta_0^3M}}q_{\alpha}+\frac{\bar{\beta}(N)}{M},
\end{align*}
we can prove (using the same method as in Lemma~\ref{lemma:BN-AsymEst}) that
\begin{align*}
\widetilde{B}_N\sim \exp\left(\frac{(\theta_1-\theta_0)^2}{8\theta_0^2}N-\frac{\sqrt{T}(\theta_1^2-\theta_0^2) q_{\alpha}}{\sqrt{8\theta_0^3}}\sqrt{M}+\bar{\beta}(N)\right)/\sqrt{M},\quad\textrm{ as }N\to\infty,
\end{align*}
and that
\begin{align*}
\bar{c}_{\alpha}(N)\big/\left(1-\bP^{N,T}_{\theta_1}(\bar{R}_{N})\right)\sim \sqrt{M}.
\end{align*}
In fact, from the expression of $\widetilde{B}_N$ in Lemma~\ref{lemma:BN-AsymEst} we know that
\begin{align*}
\lim_{N\to\infty}\frac{\widehat{c}_{\alpha}(N)\big/\left(1-\bP^{N,T}_{\theta_1}(\widehat{R}_{N})\right)} {\bar{c}_{\alpha}(N)\big/\left(1-\bP^{N,T}_{\theta_1}(\bar{R}_N)\right)}=1.
\end{align*}
Therefore,
\begin{align*}
\limsup_{N\to\infty}\frac{1-\bP^{N,T}_{\theta_1}(\bar{R}_N)}{1-\bP^{N,T}_{\theta_1}(\widehat{R}_N)} =\limsup_{N\to\infty}\frac{\bar{c}_{\alpha}(N)}{\widehat{c}_{\alpha}(N)}=\limsup_{N\to\infty}e^{\bar{\beta}(N)}<1,
\end{align*}
which concludes the proof.
\end{proof}

Finally, by Theorem~\ref{th:MorePowerTests-N}, similar to the proof of Theorem~\ref{Th:NoMostAsymKstar}, one can prove Theorem~\ref{Th:NoMostAsymKtilde}; the rejection region $(\widetilde{R}_N)$ given by \eqref{def:RNStar} cannot be asymptotically the most powerful in $\widetilde{\cK}_{\alpha}$.

\section*{Acknowledgments}
We would like to thank the anonymous referees and the associate editor for their helpful comments and suggestions which improved greatly the final manuscript.
Igor Cialenco acknowledges support from the NSF grant DMS-1211256.

{\small

\begin{appendix}

\section{Appendix}\label{sec:appendix}

In this section, for sake of convenience,  we present some known results that we use throughout the paper.
We conclude the section with the detailed proofs of some technical lemmas from the main body.

\subsection{Limit Theorems}
\begin{definition}[Ergodicity]
The stochastic process $X$ has ergodic property if there exists an (invariant) distribution $F$, such that for any measurable function $h$ with $\bE |h(\xi)|<\infty$, where $\xi$ has distribution $F$, we have the convergence
$$
\lim_{T\to\infty}\frac{1}{T}\int_0^Th(X_t)dt=\bE h(\xi)\quad a.s.
$$
If $\xi$ admits a density function $f$, we call $f$ the invariant density of $X$.
\end{definition}

Assume that $X$ satisfies the following stochastic differential equation
\begin{align}\label{eq:SDE-GeneralErogodic}
dX_t=S(X_t)dt+\sigma(X_t)dW_t,\quad X_0=x_0, \ t\ge0.
\end{align}
We will say that $X$ satisfies Conditions ($\mathcal{RP}$) if
\begin{align}
V(x)&:=\int_0^x\exp\left\{-2\int_0^y\frac{S(u)}{\sigma^2(u)}du\right\}dy\rightarrow\pm\infty,\quad\textrm{as}\quad x\rightarrow\pm\infty,\nonumber
\\
G&:=\int_{-\infty}^{\infty}\sigma^{-2}(y)\exp\left\{2\int_0^y\frac{S(u)}{\sigma^2(u)}du\right\}dy<\infty. \tag{$\mathcal{RP}$}\label{eq:CondRP}
\end{align}
It is well-known that conditions \eqref{eq:CondRP} guarantee that the process $X$ is ergodic. More precisely the following version of Law of Large Numbers holds true.

\begin{theorem}[Law of Large Numbers]\label{th:LLN}
Assume that the conditions $(\mathcal{RP)}$ are fulfilled. Then, the stochastic process $X$ with dynamics \eqref{eq:SDE-GeneralErogodic} has the ergodic property, with the invariant density
\begin{align}
f(x)=\frac{1}{G\sigma^2(x)}\exp\left\{2\int_0^x\frac{S(y)}{\sigma^2(y)}dy\right\}.
\end{align}
\end{theorem}
\noindent For the proof, see for instance \citet{KutoyantsBook2004} or \citet{GikhmanSkorokhodBook1996}, \citet{DurrettBook1996}.

We will also make use of the following version of the Cental Limit Theorem.
\begin{theorem}[Central Limit Theorem]\label{th:CLT}
 Let $h(t,\omega)$ be progressively measurable random process, and square integrable with respect to $t$ on any closed interval. Suppose that there exist a nonrandom function $\varphi_T$ and a positive constant $\rho$ such that
\begin{align}
\lim_{T\to\infty}\varphi_T^2\int_0^Th^2(t,\omega)dt=\rho^2<\infty\qquad\textrm{in probability}.
\end{align}
Then,
\begin{align}
\varphi_T\int_0^Th(t,\omega)dW_t\overset{d}\longrightarrow\mathcal{N}(0,\rho^2).
\end{align}
\end{theorem}
\noindent The proof can be found in \citet{KutoyantsBook2004}.

\subsection{Large Deviation Principle}\label{appendix:cum}

We start with a general result on large deviations of random processes. For a random process $S_t$ which may take only one infinite value, either $-\infty$ or $+\infty$, we define its cumulant generating function as follows
\begin{align}\label{eq:10}
m_t(\epsilon)=\bE\exp(\epsilon S_t),\qquad\epsilon\in\mathbb{R}.
\end{align}
Conventionally, if $\bP(S_t=-\infty)>0$, then  $m_t(\varepsilon)$ is well defined for $\epsilon\neq0$, and $m_t(\epsilon)=\infty$ for all $\epsilon<0$.
For $\epsilon=0$, we set $m_t(0)=\bP(S_t>-\infty)$.
If $\bP(S_t=\infty)>0$, then $m_t(\varepsilon)$ is also well defined for $\epsilon\neq0$,  and $m_t(\epsilon)=\infty$ for all $\epsilon>0$.
For $\epsilon=0$, we set $m_t(0)=\bP(S_t<\infty)$. Obviously, if $\bP(-\infty<S_t<\infty)=1$, then $m_t(0)=1$.

\begin{definition}\label{def:condM}
Suppose that $\lim_{t\to\infty}\varphi_t=\infty$. We say that $m_t(\epsilon)$ satisfies  condition $\textbf{(m)}$ if
\begin{enumerate}[(i)]
\item for all $\epsilon\in\mathbb{R}$, the limit
\begin{align}\label{eq:14}
\lim_{T\to\infty}\varphi_t^{-1}\ln m_t(\epsilon)=:c(\epsilon)
\end{align}
exists;
\item the function $c(\epsilon)$ is proper and convex with
\begin{align}
\epsilon_-:=\inf\{\epsilon: c(\epsilon)<\infty\}<\epsilon_+:=\sup\{\epsilon: c(\epsilon)<\infty\};
\end{align}
\item  $c(\epsilon)$ is differentiable on $(\epsilon_-,\epsilon_+)$, and
\begin{align}
\gamma_-:=\underset{\epsilon\downarrow\epsilon_-}{\lim}c'(\epsilon)<\gamma_+:=\underset{\epsilon\uparrow\epsilon_+}{\lim}c'(\epsilon).
\end{align}
\end{enumerate}
\end{definition}
Obviously, $\epsilon_-\le0$ and $\epsilon_+\ge0$. If $\epsilon_-<0<\epsilon_+$, then under condition $\textbf{(m)}$ the derivative $\gamma_0=c'(0)$ is well defined.

The following theorem is a Chernoff theorem for extended random variable $S_t$ under condition $\textbf{(m)}$ with $\epsilon_-<0<\epsilon_+$
(for details see for instance \citet{Linckov1999}).
\begin{theorem}\label{th:Chernoff}\textbf{(Large Deviation)}
If condition $\textbf{(m)}$ is satisfied with $\epsilon_-<0<\epsilon_+$, then the following statements hold:\\
(1) If $\gamma_0<\gamma_+$, then for all $\gamma\in(\gamma_0,\gamma_+)$
\begin{align}
\lim_{T\to\infty}\varphi_t^{-1}\ln \bP(\varphi_t^{-1}S_t>\gamma)=\lim_{T\to\infty}\varphi_t^{-1}\ln \bP(\varphi_t^{-1}S_t\ge\gamma)=-I(\gamma);
\end{align}
(2) If $\gamma_-<\gamma_0$, then for all $\gamma\in(\gamma_-,\gamma_0)$
\begin{align}
\lim_{T\to\infty}\varphi_t^{-1}\ln \bP(\varphi_t^{-1}S_t<\gamma)=\lim_{T\to\infty}\varphi_t^{-1}\ln \bP(\varphi_t^{-1}S_t\le\gamma)=-I(\gamma).
\end{align}
Here, $I(\gamma)$ denotes the Legendre-Fenchel transform of $c(\epsilon)$, that is $I(\gamma)=\sup_\epsilon\{\epsilon\gamma-c(\epsilon)\}$.
\end{theorem}

\subsection{Sharp Large Deviation for OU process}\label{appendix:SLDforOU}

Consider the Ornstein-Uhlenbeck process
\begin{align*}
dX_t=\theta X_t+\sigma dW_t,\qquad X(0)=0,
\end{align*}
where $\theta$ is strictly negative. Define
\begin{align*}
\mathcal{Z}_T(a,b)=&a\int_0^TX_tdX_t+b\int_0^TX_t^2dt,
\\
\mathcal{L}_T(a,b)=&T^{-1}\ln\bE\left[\exp\left(\mathcal{Z}_T(a,b)\right)\right].
\end{align*}
Then, we have the following result.
\begin{proposition}\label{lemma:BercueSLD}
Set $\Delta=\{(a,b)\in\mathbb{R}^2|\theta^2-2b\sigma^2>0\textrm{ and }\theta+a\sigma^2<\sqrt{\theta^2-2b\sigma^2}\}$ and let $\rho(b)=\sqrt{\theta^2-2b\sigma^2}$. 
Then, for all $(a,b)\in\Delta$, 
\begin{align*}
\mathcal{L}_T(a,b)=\mathcal{L}(a,b)+T^{-1}\mathcal{H}(a,b)+T^{-1}\mathcal{R}_T(a,b)
\end{align*}
with
\begin{align*}
\mathcal{L}(a,b)=&-\frac{1}{2}(a\sigma^2+\theta+\rho(b)),
\\
\mathcal{H}(a,b)=&-\frac{1}{2}\ln\left(\frac{1}{2}(1-(a\sigma^2+\theta)\rho^{-1}(b))\right),
\\
\mathcal{R}_T(a,b)=&-\frac{1}{2}\ln\left(1+\frac{1+(a\sigma^2+\theta)\rho^{-1}(b)}{1-(a\sigma^2+\theta)\rho^{-1}(b)}e^{-2T\rho(b)}\right).
\end{align*}
\end{proposition}
We cite this result from \citet{BercuRouault2001}, in which $\sigma=1$. It can be also proved by evaluating  the cumulant generating function of $\mathcal{Z}_T(a,b)$ using Feynman-Kac formula, similar to the method we used to get \eqref{eq:34} (takeing $N=1$).

\subsection{Proofs of some auxiliary results}\label{TechLemma}

We start with a simple result about the form of remainder from Taylor expansion of an analytic function, that is convenient  for our purposes.
The proof is done by standard technics, that can be found for example in \citet[Theorem 8.7]{BeckBook2002} or \citet{Palka1991}.

\begin{proposition}\label{prop:taylorExpension}
Suppose that the function $f(z)$ is analytic in a domain $U$ on complex plain, with $z$ and $z_0$ contained in $U$ together with a circle $\widetilde{\mathbb{S}}$.
 Then, we have a Taylor expansion around $z_0$
\begin{align*}
f(z)=\sum_{j=0}^k\frac{f^{(j)}(z_0)}{j!}(z-z_0)^j+\mathcal{E}_k(z),
\end{align*}
with remainder
\begin{align*}
\mathcal{E}_k(z)=\frac{(z-z_0)^{k+1}}{2\pi i}\int_{\widetilde{\mathbb{S}}}\frac{f(w)dw}{(w-z_0)^{k+1}(w-z)}.
\end{align*}
\end{proposition}

Next we prove a version of CLT for martingales, suitable for our needs.

\begin{lemma}\label{lemma:YN} Assume that $u_k, \ k\geq 1$, are stochastic processes with dynamics \eqref{eq:OU-Fourier}, and zero initial data. Then,
\begin{enumerate}[(i)]
\item \begin{equation}\label{eq:AsymNorm-YN}
M^{-1/2}\sum_{k=1}^N \lambda_k^{2\beta+\gamma} \int_0^T u_k dw_k \overset{d} \longrightarrow \xi,   \quad N\to\infty,
\end{equation}
where $M=\sum_{k=1}^N\lambda_k^{2\beta}$, and $\xi$ is a Gaussian random variable with mean zero and variance $\frac{\sigma^2 T}{2\theta_0}$

\item
\begin{align*}
N^{1/2}\left(\frac{1}{N}\sum_{k=1}^N\mathcal{X}_k-1\right)\overset{d} \longrightarrow \mathcal{N}(0,2),\quad\textrm{ as $N\to\infty$},
\end{align*}
where $\mathcal{X}_k:=2\theta_0\sigma^{-2}\left(1-e^{-2\theta_0\lambda_k^{2\beta} T}\right)^{-1}\lambda_k^{2\beta+2\gamma}u_k^2(T)$, for $k\geq 1$.
\end{enumerate}
\end{lemma}

\begin{proof}
Take
\begin{align*}
\xi_k=\lambda_k^{2\beta+\gamma}\int_0^Tu_kdw_k.
\end{align*}
Note that $\mathbf{E}\xi_k=0$, and by direct calculations,
\begin{align}\label{eq:xiVar}
\textrm{Var}\xi_k=&\mathbf{E}\xi_k^2=\lambda_k^{4\beta+2\gamma}\int_0^T\mathbf{E}u_k^2dt\notag
\\
=&\frac{\sigma^2T}{2\theta}\lambda_k^{2\beta}-\frac{\sigma^2}{4\theta^2}\left(1-e^{-2\theta\lambda_k^{2\beta}T}\right)\sim k^{2\beta/d}.
\end{align}
Hence, using \citet[Theorem 2.4]{Lototsky2009Survey},  we deduce that 
\begin{align}\label{eq:AsymNorm-xiN}
\lim_{N\to\infty}\frac{\sum_{k=1}^N\lambda_k^{2\beta+\gamma}\int_0^T u_kdw_k}{\left(\sum_{k=1}^N\lambda_k^{4\beta+2\gamma}\int_0^T\mathbf{E} u_k^2dt\right)^{1/2}}\overset{d}{\longrightarrow}\mathcal{N}(0,1).
\end{align}
By \eqref{eq:xiVar} and basic calculus we have
\begin{align*}
\lim_{N\to\infty}M^{-1}\sum_{k=1}^N\lambda_k^{4\beta+2\gamma}\int_0^T\mathbf{E} u_k^2dt=\frac{\sigma^2T}{2\theta}.
\end{align*}
This, combined with \eqref{eq:AsymNorm-xiN}, concludes the first claim.

By direct evaluations, we obtain that  $u_k(T) \overset{d}{\sim} \mathcal{N}(0,\sigma_k^2)$, where
$$
\sigma_k^2=\frac{\sigma^2}{2\theta_0}\lambda_k^{-2\beta-2\gamma}\left(1-e^{-2\theta_0\lambda_k^{2\beta} T}\right).
$$
By definition $\mathcal{X}_k\overset{d}{\sim} \chi^2(1)$, for all $k\ge1$.
Since $u_k(T)$'s are independent, by Central Limit Theorem, the second claim follows immediately.
\end{proof}

In what follows we present the proofs of three technical lemmas from Section~\ref{sec:AsympMethodN}.
\begin{proof}[Proof of Lemma~\ref{lemma:CharExpan-VN}]
From \eqref{eq:MeasureChange2} we immediately have
\begin{align*}
\widetilde{\Psi}_N(u)=\exp\left[-\frac{iu\widetilde{\mathcal{H}}'(\widetilde{\epsilon}_{\eta})N}{\zeta_{\eta}\sqrt{M}}-\frac{iu\eta \sqrt{M}}{\zeta_{\eta}}+M\left(\mathcal{L}_N\left(\widetilde{\epsilon}_{\eta}+\frac{iu}{\zeta_{\eta}\sqrt{M}}\right) -\mathcal{L}_N(\widetilde{\epsilon}_{\eta})\right)\right].
\end{align*}
Let $h=\frac{iu}{\zeta_{\eta}\sqrt{M}}$, then the above $\widetilde{\Psi}_N$ can be rewritten as
\begin{align*}
\widetilde{\Psi}_N(u)=J_1^N(u)J_2^N(u)J_3^N(u),
\end{align*}
where
\begin{align*}
J_1^N(u):=&\exp\left[-\frac{iu\eta\sqrt{M}}{\zeta_{\eta}}+M\left(\widetilde{\mathcal{L}} \left(\widetilde{\epsilon}_{\eta}+h\right)-\widetilde{\mathcal{L}}(\widetilde{\epsilon}_{\eta})\right)\right],
\\
J_2^N(u):=&\exp\left[\widetilde{\mathcal{R}}_N(\widetilde{\epsilon}_{\eta}+h)-\widetilde{\mathcal{R}}_N(\widetilde{\epsilon}_{\eta})\right],
\\
J_3^N(u):=&\exp\left[N\left(\widetilde{\mathcal{H}}(\widetilde{\epsilon}_{\eta}+h)-\widetilde{\mathcal{H}} (\widetilde{\epsilon}_{\eta})-\widetilde{\mathcal{H}}'(\widetilde{\epsilon}_{\eta})h\right)\right].
\end{align*}
By Taylor  expansion formula, we have that
\begin{align*}
\ln J_1^N(u)=-\frac{u^2}{2} +M\sum_{k=3}^{n+3}\left(\frac{iu}{\zeta_{\eta}\sqrt{M}}\right)^k\frac{\widetilde{\mathcal{L}}^{(k)}(\widetilde{\epsilon}_{\eta})}{k!}.
\end{align*}
Hence, by the same method as in Lemma~\ref{lemma:CharExpan-VT}, we deduce that
\begin{align}\label{eq:L-Expan}
J_1^N(u)=\exp\left(-\frac{u^2}{2}\right) \left[1-\frac{i\widetilde{\mathcal{L}}^{(3)}(\widetilde{\epsilon}_{\eta})u^3}{6\zeta_{\eta}^3\sqrt{M}}+O\left(\frac{|u|^4+|u|^6}{M}\right)\right],
\end{align}
where the remainder is uniform in $u$ as long as $|u|=O(M^{1/6})$. Similarly, we get
\begin{align*}
\ln J_3^N(u)=N\cdot O(h^2)=O\left(\frac{u^2N}{M}\right),
\end{align*}
where the remainder is uniform in $u$ as long as $|u|=O(\sqrt{M})$, and also
\begin{align}\label{eq:H-Expan}
J_3^N(u)=&1+ N\left(\widetilde{\mathcal{H}}(\widetilde{\epsilon}_{\eta}+h)-\widetilde{\mathcal{H}} (\widetilde{\epsilon}_{\eta})-\widetilde{\mathcal{H}}'(\widetilde{\epsilon}_{\eta})h\right)+O\left(\frac{u^4N^2}{M^2}\right)\notag
\\
=&1-\frac{u^2N}{2\zeta_{\eta}^2M}\widetilde{\mathcal{H}}''(\widetilde{\epsilon}_{\eta})+N\cdot O(h^3)+O\left(\frac{u^4N^2}{M^2}\right)\notag
\\
=&1-\frac{u^2N}{2\zeta_{\eta}^2M}\widetilde{\mathcal{H}}''(\widetilde{\epsilon}_{\eta})+O\left(\frac{|u|^3N}{M^{3/2}}+\frac{u^4N^2}{M^2}\right),
\end{align}
where the remainder is uniform in $u$ as long as $|u|=O(\sqrt{M/N})$.

Define
\begin{align*}
g(h):=\frac{1-\mathcal{D}(h)} {1+\mathcal{D}(h)},\qquad a_k(h):=2\lambda_k^{2\beta}T\sqrt{\theta_1^2+(\theta_1^2-\theta_0^2)h}.
\end{align*}
In the following, $\mathbb{S}$ denotes any circle around the singularities of the integrand. By Proposition~\ref{prop:taylorExpension} we have
\begin{align}\label{eq:RNVariate}
\widetilde{\mathcal{R}}_N(\widetilde{\epsilon}_{\eta}+h)=&-\frac{1}{2}\sum_{k=1}^N\frac{g(\widetilde{\epsilon}_{\eta}+h)\exp(-a_k(\widetilde{\epsilon}_{\eta}+h))}{2\pi i}G_k(\widetilde{\epsilon}_{\eta}+h)\notag
\\
=&-(4\pi i)^{-1}g(\widetilde{\epsilon}_{\eta}+h)\sum_{k=1}^N\exp(-a_k(\widetilde{\epsilon}_{\eta}+h))G_k(\widetilde{\epsilon}_{\eta}+h),
\end{align}
where
\begin{align*}
G_k(\widetilde{\epsilon}_{\eta}+h)=\int_{\mathbb{S}}\frac{\ln zdz}{(z-1)\left(z-1-g(\widetilde{\epsilon}_{\eta}+h)\exp(-a_k(\widetilde{\epsilon}_{\eta}+h))\right)}.
\end{align*}
Furthermore, $G_k(\cdot)$'s are differentiable and from the definition of $a_k$ we know that, for $T$ large enough, $\{G_k\}$ and $\{G_k'\}$ are uniformly locally bounded.
Again using Proposition~\ref{prop:taylorExpension}, we get the following three expansions,
\begin{align}\label{eq:RNExpan1}
G_k(\widetilde{\epsilon}_{\eta}+h)=&G_k(\widetilde{\epsilon}_{\eta})+G_k'(\widetilde{\epsilon}_{\eta})h+\frac{h^2}{2\pi i}\int_{\mathbb{S}}\frac{G_k(z)dz} {(z-\widetilde{\epsilon}_{\eta})^2\left(z-\widetilde{\epsilon}_{\eta}-h\right)}\notag
\\
=&G_k(\widetilde{\epsilon}_{\eta})+\frac{iu}{\zeta_{\eta}\sqrt{M}}G_k'(\widetilde{\epsilon}_{\eta})+O\left(\frac{u^2}{M}\right), \\
\exp(-a_k(\widetilde{\epsilon}_{\eta}+h))=&e^{-a_k(\widetilde{\epsilon}_{\eta})}-e^{-a_k(\widetilde{\epsilon}_{\eta})}a_k'(\widetilde{\epsilon}_{\eta})h+\frac{h^2}{2\pi i}\int_{\mathbb{S}}\frac{e^{-a_k(z)}dz} {(z-\widetilde{\epsilon}_{\eta})\left(z-\widetilde{\epsilon}_{\eta}-h\right)}\notag
\\
=&e^{-a_k(\widetilde{\epsilon}_{\eta})}-e^{-a_k(\widetilde{\epsilon}_{\eta})}a_k'(\widetilde{\epsilon}_{\eta})h+\frac{e^{-a_k(\widetilde{\epsilon}_{\eta})/2}h^2}{2\pi i}\int_{\mathbb{S}}\frac{e^{a_k(\widetilde{\epsilon}_{\eta})/2-a_k(z)}dz} {(z-\widetilde{\epsilon}_{\eta})\left(z-\widetilde{\epsilon}_{\eta}-h\right)}\notag
\\
=&e^{-a_k(\widetilde{\epsilon}_{\eta})}-\frac{iu}{\zeta_{\eta}\sqrt{M}}e^{-a_k(\widetilde{\epsilon}_{\eta})}a_k'(\widetilde{\epsilon}_{\eta}) +e^{-a_k(\widetilde{\epsilon}_{\eta})/2}O\left(\frac{u^2}{M}\right),\label{eq:RNExpan2} \\
g(\widetilde{\epsilon}_{\eta}+h)=&g(\widetilde{\epsilon}_{\eta})+g'(\widetilde{\epsilon}_{\eta})h+O\left(h^2\right) =g(\widetilde{\epsilon}_{\eta})+\frac{iu}{\zeta_{\eta}\sqrt{M}}g'(\widetilde{\epsilon}_{\eta})+O\left(\frac{u^2}{M}\right), \label{eq:RNExpan3}
\end{align}
where in the first two expansion the remainder is uniform in $k$ and $u$, and in the last expansion the remainder is uniform in $u$,  as long as $|u|=O(\sqrt{M})$.
Using \eqref{eq:RNExpan1}--\eqref{eq:RNExpan3} into \eqref{eq:RNVariate}, we obtain
\begin{align*}
\widetilde{\mathcal{R}}_N(\widetilde{\epsilon}_{\eta}+h)=&-(4\pi i)^{-1}\left(g(\widetilde{\epsilon}_{\eta})+\frac{iu}{\zeta_{\eta}\sqrt{M}}g'(\widetilde{\epsilon}_{\eta}) +O\left(\frac{u^2}{M}\right)\right)
\\
&\sum_{k=1}^N\left(G_k(\widetilde{\epsilon}_{\eta})e^{-a_k(\widetilde{\epsilon}_{\eta})}+ \left(G_k'(\widetilde{\epsilon}_{\eta})e^{-a_k(\widetilde{\epsilon}_{\eta})}-G_k(\widetilde{\epsilon}_{\eta}) e^{-a_k(\widetilde{\epsilon}_{\eta})}a_k'(\widetilde{\epsilon}_{\eta})\right) \frac{iu}{\zeta_{\eta}\sqrt{M}}+e^{-a_k(\widetilde{\epsilon}_{\eta})/2}O\left(\frac{u^2}{M}\right)\right)
\\
=&C_0^N+C_\eta\frac{u}{4\pi\zeta_{\eta}\sqrt{M}}+O\left(\frac{|u|+|u|^2}{M}\right),
\end{align*}
where the remainder is uniform in $k$ and $u$, for $|u|=O(\sqrt{M})$, and where
\begin{align*}
C_0^N=&-(4\pi i)^{-1}g(\widetilde{\epsilon}_{\eta})\sum_{k=1}^N G_k(\widetilde{\epsilon}_{\eta})e^{-a_k(\widetilde{\epsilon}_{\eta})},
\\
C_\eta=&-g'(\widetilde{\epsilon}_{\eta})\sum_{k=1}^\infty G_k(\widetilde{\epsilon}_{\eta})e^{-a_k(\widetilde{\epsilon}_{\eta})}- g(\widetilde{\epsilon}_{\eta})\sum_{k=1}^\infty G_k'(\widetilde{\epsilon}_{\eta})e^{-a_k(\widetilde{\epsilon}_{\eta})}- G_k(\widetilde{\epsilon}_{\eta})e^{-a_k(\widetilde{\epsilon}_{\eta})}a_k'(\widetilde{\epsilon}_{\eta}).
\end{align*}
Using the definition of $G_k$, we find that $C_0^N=\widetilde{\mathcal{R}}_N(\widetilde{\epsilon}_{\eta})$.
Therefore,
\begin{align}\label{eq:R-Expan}
J_2^N(u)=&1+ \frac{C_\eta u}{4\pi\zeta_{\eta}\sqrt{M}}+O\left(\frac{|u|+|u|^2}{M}\right),
\end{align}
where the remainder is uniform in $u$, for $|u|=O(\sqrt{M})$.

Finally, combining \eqref{eq:L-Expan}, \eqref{eq:H-Expan} and \eqref{eq:R-Expan} we immediately have \eqref{eq:CharExpan-VN}, and this concludes the proof.

\end{proof}

\begin{proof}[Proof of Lemma~\ref{lemma:BN-AsymEst}]
As in the proof of Lemma~\ref{lemma:BT-AsymEst}, we first need to show that the characteristic function $\widetilde{\Psi}_N$ can be controlled uniformly by an integrable function for large $N$.
For simplicity, we use $C$ to denote constants independent of $u$, $T$ and $N$, whose value may change from line to line.

For $\eta=-\frac{(\theta_1-\theta_0)^2}{4\theta_0}T$,  we have
\begin{align*}
\left|J_2^N(u)\right|=&\left|\exp\left[\widetilde{\mathcal{R}}_N(\widetilde{\epsilon}_{\eta}+h) -\widetilde{\mathcal{R}}_N(\widetilde{\epsilon}_{\eta})\right]\right|
\\
=&\left|\prod_{k=1}^N\left(1+\frac{1-\mathcal{D}(-1+h)}{1+\mathcal{D}(-1+h)}\exp\left(-a_k(-1+h)\right)\right)^{-1/2}\right|
\\
\le&\prod_{k=1}^N\left(1-\left|\frac{1-\mathcal{D}(-1+h)}{1+\mathcal{D}(-1+h)}\right|\left|\exp\left(-a_k(-1+h)\right)\right|\right)^{-1/2}
\\
=&\prod_{k=1}^N\left(1-\frac{\left|D_1^2(h)-D_2^2(h)\right|}{\left|D_1(h)+D_2(h)\right|^2}\left|\exp\left(-a_k(-1+h)\right)\right|\right)^{-1/2},
\end{align*}
where $h=\frac{iu}{\zeta_{\eta}\sqrt{M}}$,
$D_1(h):=\theta_0+(\theta_1-\theta_0)h$, and $D_2(h):=\sqrt{\theta_0^2+(\theta_1^2-\theta_0^2)h}.$
Using the definitions of $D_j$'s, it is not hard to verify that
\begin{align}\label{eq:D1D2Est}
D_2(h)=&\left(\left(\left(\theta_0^4+(\theta_1^2-\theta_0^2)^2|h|^2\right)^{1/2}+\theta_0^2\right)/2\right)^{1/2}\notag
\\
&+i\textrm{sign}(u) \left(\left(\left(\theta_0^4+(\theta_1^2-\theta_0^2)^2|h|^2\right)^{1/2}-\theta_0^2\right)/2\right)^{1/2},\notag
\\
\left|D_1^2(h)-D_2^2(h)\right|=&2\theta_0(\theta_1-\theta_0)|h|\sqrt{1+|h|^2},\qquad \left|D_1(h)+D_2(h)\right|^2\ge4\theta_0^2,
\end{align}
and also, by using the fact that $\sqrt{2x+2y}\ge\sqrt{x}+\sqrt{y}$, we have the following estimate
\begin{align}\label{eq:akExEst}
\left|\exp\left(-a_k(-1+h)\right)\right|=&\exp\left[-\lambda_k^{2\beta}T \left(2\left(\theta_0^4+(\theta_1^2-\theta_0^2)^2|h|^2\right)^{1/2}+2\theta_0^2\right)^{1/2}\right]\notag
\\
\le&\exp\left[-\lambda_k^{2\beta}T \left(C+C|h|\right)^{1/2}\right]\le\exp\left[-CT\lambda_k^{2\beta} \left(1+\sqrt{|h|}\right)\right].
\end{align}
Therefore, by all the above analysis and Taylor's expansion, we have
\begin{align*}
\ln\left|J_2^N(u)\right|\le&-\frac{1}{2}\sum_{k=1}^N\ln\left(1-\frac{\left|D_1^2(h)-D_2^2(h)\right|} {\left|D_1(h)+D_2(h)\right|^2}\left|\exp\left(-a_k(-1+h)\right)\right|\right)
\\
\le&-\frac{1}{2}\sum_{k=1}^N\ln\left(1-C|h|\sqrt{1+|h|^2}\exp\left[-CT\lambda_k^{2\beta} \left(1+\sqrt{|h|}\right)\right]\right)
\\
\le&-\frac{1}{2}\sum_{k=1}^N\ln\left(1-Ce^{-CT\lambda_k^{2\beta}}\right)= C\sum_{k=1}^N\left(1-\delta_k\right)^{-1}e^{-CT\lambda_k^{2\beta}},
\end{align*}
where $\delta_k$'s  are some positive numbers between $0$ and $Ce^{-CT\lambda_k^{2\beta}}$, and it easy to observe that if we choose $T$ large enough, then $1-\delta_k$ are uniformly bounded up away from $0$.
This implies that for large enough $T$ we have the following estimate
\begin{align}\label{eq:J2Est}
\left|J_2^N(u)\right|\le&C\sum_{k=1}^Ne^{-CT\lambda_k^{2\beta}}\le C.
\end{align}
As far as $J_3^N(u)$, just notice that
\begin{align*}
\left|J_3^N(u)\right|=&\left|\exp\left[N\left(\widetilde{\mathcal{H}}(\widetilde{\epsilon}_{\eta}+h)-\widetilde{\mathcal{H}} (\widetilde{\epsilon}_{\eta})-\widetilde{\mathcal{H}}'(\widetilde{\epsilon}_{\eta})h\right)\right]\right|
\\
=&\left|\exp\left[N\left(\widetilde{\mathcal{H}}(-1+h)-\widetilde{\mathcal{H}}(-1)\right)\right]\right|
\\
=&\left|\left(\frac{1}{2}+\frac{1}{2}\mathcal{D}(-1+h)\right)^{-N/2}\right| =\left(\frac{2|D_2(h)|}{\left|D_1(h)+D_2(h)\right|}\right)^{N/2}.
\end{align*}
Using the fact that $(1+x)^{1/4}\le 1+x/4$, we obtain
\begin{align}\label{eq:D2normEst}
\theta_0\le|D_2(h)|=&\left(\theta_0^4+(\theta_1^2-\theta_0^2)^2|h|^2\right)^{1/4}\le\theta_0+Cu^2/M.
\end{align}
Then, we have the estimate
\begin{align}\label{eq:J3Est}
\left|J_3^N(u)\right|\le&\left(1+Cu^2/M\right)^{N/2}.
\end{align}
For estimating $J_1^N(u)$, we first need to refer to the proof of Lemma~\ref{lemma:CharExpan-VT}. Define
\begin{align*}
\Psi_{V_T}(u)=J_1^T(u)J_2^T(u)J_3^T(u),
\end{align*}
where
\begin{align*}
J_1^T(u)=&\exp\left[-\frac{iu\bar{\eta}\sqrt{T}}{\varsigma_{\bar{\eta}}}+T\left(\mathcal{L}\left(\epsilon_{\bar{\eta}} +\bar{h}\right)-\mathcal{L}(\epsilon_{\bar{\eta}})\right)\right],\qquad \bar{h}=\frac{iu}{\varsigma_{\bar{\eta}}\sqrt{T}}
\\
J_2^T(u)=&\exp\left[\mathcal{H}(\epsilon_{\bar{\eta}}+\bar{h})-\mathcal{H}(\epsilon_{\bar{\eta}})\right],\qquad J_3^T(u)=\exp\left[\mathcal{R}_N(\epsilon_{\bar{\eta}}+\bar{h})-\mathcal{R}_N(\epsilon_{\bar{\eta}})\right].
\end{align*}
To avoid the confliction of the notations, here we use $\bar{\eta}$, $\bar{h}$ instead of $\eta$, $h$.
By repeatedly using the fact that $\sqrt{1+x}\le 1+x/2$, we obtain
\begin{align*}
\left|D_1(\bar{h})+D_2(\bar{h})\right|\le2\theta_0+Cu^2/T.
\end{align*}
Now take $N=1$ and $M=\lambda_1^{2\beta}$. Then, by the above inequality and \eqref{eq:D2normEst} we know that, for $\bar{\eta}=-\frac{(\theta_1-\theta_0)^2}{4\theta_0}\lambda_1^{2\beta}$,
\begin{align*}
|J_2^T(u)|^{-1}=\left(\frac{\left|D_1(\bar{h})+D_2(\bar{h})\right|}{2|D_2(\bar{h})|}\right)^{1/2}\le\left(1+Cu^2/T\right)^{1/2}.
\end{align*}
By \eqref{eq:D1D2Est} and \eqref{eq:akExEst}, we have
\begin{align*}
|J_3^T(u)|^{-1}\le&\left(1+\left|\frac{1-\mathcal{D}(-1+\bar{h})}{1+\mathcal{D}(-1+\bar{h})}\right|\left| \exp\left(-a_1(-1+\bar{h})\right)\right|\right)^{1/2}
\\
=&\left(1+\frac{\left|D_1^2(\bar{h})-D_2^2(\bar{h})\right|}{\left|D_1(\bar{h})+D_2(\bar{h})\right|^2}\left| \exp\left(-a_k(-1+\bar{h})\right)\right|\right)^{1/2}\le C.
\end{align*}
Therefore, by \eqref{eq:CharExControl} and the above estimates we have
\begin{align*}
|J_1^T(u)|=|\Psi_{V_T}(u)||J_2^T(u)|^{-1}|J_3^T(u)|^{-1}\le C\left(1+\frac{\mu_1 u^2}{T}\right)^{-\mu_2 T}\left(1+\frac{Cu^2}{T}\right)^{1/2}.
\end{align*}
Then, for any $T'>0$, we have that
\begin{align*}
|J_1^T(u)|=&\left|\exp\left[T\lambda_1^{2\beta}\left(D_1(\bar{h})-D_2(\bar{h})\right)/2\right]\right|
\\
=&\left|\exp\left[TT'\left(D_1(iu/\varsigma_{\bar{\eta}}\sqrt{T})-D_2(iu/\varsigma_{\bar{\eta}}\sqrt{T})\right)/2\right]\right| ^{\lambda_1^{2\beta}/T'}\le C\left(1+\frac{\mu_1 u^2}{T}\right)^{-\mu_2 T}\left(1+\frac{Cu^2}{T}\right)^{1/2}.
\end{align*}
Since the above inequality holds true for all  parameters, as long as $T$ is large enough, we do the following substitution and rescaling
\begin{align*}
T=M,\qquad T'=T,\qquad u=\varsigma_{\bar{\eta}}u/\zeta_{\eta}
\end{align*}
that yields that for some constant $\widetilde{\mu}_1$ and $\widetilde{\mu}_2$ and large enough $M$
\begin{align*}
\left|\exp\left[MT\left(D_1(h)-D_2(h)\right)/2\right]\right| ^{\lambda_1^{2\beta}/T}\le C\left(1+\frac{\widetilde{\mu}_1 u^2}{M}\right)^{-\widetilde{\mu}_2 M}\left(1+\frac{Cu^2}{M}\right)^{1/2}.
\end{align*}
Since $|J_1^N(u)|=\left|\exp\left[MT\left(D_1(h)-D_2(h)\right)/2\right]\right|$, we also have
\begin{align}\label{eq:J1Est}
|J_1^N(u)|\le C\left(1+\frac{\widetilde{\mu}_1 u^2}{M}\right)^{-\widetilde{\mu}_2TM/\lambda_1^{2\beta}}\left(1+\frac{Cu^2}{M}\right)^{T/2\lambda_1^{2\beta}}.
\end{align}
Now combining \eqref{eq:J2Est}--\eqref{eq:J1Est}, for large enough $N$, we obtain the followings estimates
\begin{align}\label{eq:ExControl-PsiN}
|\widetilde{\Psi}_N|\le C\left(1+\frac{\widetilde{\mu}_1 u^2}{M}\right)^{-\widetilde{\mu}_2TM/\lambda_1^{2\beta}}\left(1+\frac{Cu^2}{M}\right)^{N/2+T/2\lambda_1^{2\beta}}\le C\left(1+\frac{\widetilde{\mu}_1 u^2}{M}\right)^{-\widetilde{\mu}_2TM/2\lambda_1^{2\beta}}.
\end{align}
This implies that $\widetilde{\Psi}_N$ is square integrable, and it can be controlled uniformly by an integrable function.

Follow similar procedures as in Lemma~\ref{lemma:BT-AsymEst},  and by some direct evaluations, for $\eta=-\frac{(\theta_1-\theta_0)^2}{4\theta_0}T$,
we have the following identities
\begin{align*}
\widetilde{B}_N=&\bE_N\left(\exp\left(-\widetilde{\epsilon}_{\eta}\zeta_{\eta}\sqrt{M}\widetilde{V}_N-\widetilde{\epsilon}_{\eta} \widetilde{\mathcal{H}}'(\widetilde{\epsilon}_{\eta})N\right)\1_{\{\widetilde{V}_N\le\widetilde{\eta}\sqrt{M}/\zeta_{\eta}\}}\right)
\\
=&\exp\left(-\widetilde{\epsilon}_{\eta} \widetilde{\mathcal{H}}'(\widetilde{\epsilon}_{\eta})N\right)\int_{-\infty}^{\widetilde{\eta}\sqrt{M}/\zeta_{\eta}} \exp\left(-\widetilde{\epsilon}_{\eta}\zeta_{\eta}\sqrt{M}v\right)\left((2\pi)^{-1}\int_{\mathbb{R}}\exp(-iuv)\widetilde{\Psi}_N(u)du\right)dv
\\
=&\frac{\exp\left(-\widetilde{\epsilon}_{\eta} \widetilde{\mathcal{H}}'(\widetilde{\epsilon}_{\eta})N\right)}{2\pi}\int_{\mathbb{R}}\widetilde{\Psi}_N(u)\left(\int_{-\infty}^{\widetilde{\eta}\sqrt{M}/\zeta_{\eta}} \exp\left(-\widetilde{\epsilon}_{\eta}\zeta_{\eta}\sqrt{M}v-iuv\right)dv\right)du
\\
=&-\frac{\exp\left(-\widetilde{\epsilon}_{\eta} \widetilde{\mathcal{H}}'(\widetilde{\epsilon}_{\eta})N-\widetilde{\epsilon}_{\eta}\widetilde{\eta}M\right)}{2\pi\widetilde{\epsilon}_{\eta}\zeta_{\eta}\sqrt{M}} \int_{\mathbb{R}}\left(1+\frac{iu}{\widetilde{\epsilon}_{\eta}\zeta_{\eta}\sqrt{M}}\right)^{-1} \exp\left(-iu\widetilde{\eta}\sqrt{M}/\zeta_{\eta}\right)\widetilde{\Psi}_N(u)du
\\
=&\frac{\exp\left(\frac{(\theta_1-\theta_0)^2}{8\theta_0^2}N-\frac{\sqrt{T}(\theta_1^2-\theta_0^2) q_{\alpha}}{\sqrt{8\theta_0^3}}\sqrt{M}-\frac{\sqrt{T}(\theta_1^2-\theta_0^2)}{\sqrt{8\theta_0^3}}\delta\right)} {2\pi\zeta_{\eta} \sqrt{M}}
\\
&\qquad\qquad\times\int_{\mathbb{R}}\left(1-\frac{iu}{\zeta_{\eta}\sqrt{M}}\right)^{-1} \exp\left(iu(q_{\alpha}+\delta/\sqrt{M})\right)\widetilde{\Psi}_N(u)du,
\end{align*}
where
\begin{align*}
\widetilde{\eta}=&-\frac{\widetilde{\mathcal{H}}'(\widetilde{\epsilon}_{\eta})N}{M}+\frac{(\theta_1-\theta_0)^2N}{8\theta_0^2M} -\frac{\sqrt{T}(\theta_1^2-\theta_0^2)}{\sqrt{8\theta_0^3M}}q_{\alpha} -\frac{\sqrt{T}(\theta_1^2-\theta_0^2)}{\sqrt{8\theta_0^3}M}\delta
\\
=&-\frac{\sqrt{T}(\theta_1^2-\theta_0^2)}{\sqrt{8\theta_0^3M}}q_{\alpha} -\frac{\sqrt{T}(\theta_1^2-\theta_0^2)}{\sqrt{8\theta_0^3}M}\delta,
\end{align*}
since $\widetilde{\mathcal{H}}'(\widetilde{\epsilon}_{\eta})=(\theta_1-\theta_0)^2/8\theta_0^2$.
Using \eqref{eq:CharExpan-VN} and Dominated  Convergent Theorem, we conclude that the last integral converges, as $N\to\infty$, to a finite non-zero constant, and this concludes the proof.
\end{proof}

\begin{proof}[Proof of Lemma~\ref{lemma:ProbExpan-IN}]
Define $I^N:=Y^N-X^N$, and let us use $\widetilde{\Psi}_{I^N}$ to denote the characteristic function of $I^N$ under probability measure $\bP^{N,T}_{\theta_0}$.
By the same approach as in Lemma~\ref{lemma:ProbExpan-IT}, one can show that, for $\eta=-\frac{(\theta_1-\theta_0)^2}{4\theta_0}T$,
\begin{align*}
\widetilde{\Psi}_{I^N}(u)=\widetilde{\Psi}_N(-u).
\end{align*}
Then, by taking $\eta=-\frac{(\theta_1-\theta_0)^2}{4\theta_0}T$ in \eqref{eq:CharExpan-VN}, we get
\begin{align*}
\widetilde{\Psi}_{I^N}(u)=e^{-u^2/2}+\psi_1(u)M^{-1/2}+\psi_2(u)NM^{-1}+\mathbf{r}_N(u),
\end{align*}
where
\begin{align*}
\psi_1(u)=&-\left(\frac{iu^3}{\sqrt{2\theta_0T}}+\frac{(\theta_1-\theta_0)iu}{\sqrt{8\theta_0T}(\theta_1+\theta_0)} \left(\sum_{k=1}^{\infty}e^{-2\theta_0T\lambda_k^{2\beta}}\right)\right)e^{-u^2/2},
\\
\psi_2(u)=&\frac{(\theta_1-\theta_0)(5\theta_1^2+6\theta_1\theta_0-3\theta_0^2)} {8\theta_0(\theta_1+\theta_0)(\theta_1^2-\theta_0^2)T}u^2e^{-u^2/2},
\\
\mathbf{r}_N(u)=&e^{-u^2/2}O\left(\frac{|u|+|u|^6}{M}+\frac{|u|^3N}{M^{3/2}}+\frac{u^4N^2}{M^2}\right).
\end{align*}
Note that the high order term $\mathbf{r}_N(u) e^{u^2/2}$ is uniformly bounded for $|u|=O(N^\nu)$.

First, we will prove \eqref{eq:ProbExpan-IN} for the case $\delta=0$, for which we will suppress the index $\delta$ in $\Phi_1^\delta(x)$, $\Phi_2^\delta(x)$ and $\mathfrak{R}_N^{\delta}(x)$. Put
\begin{align}\label{eq:RemainDefN}
\mathbf{R}_N(x):=\bP^{N,T}_{\theta_0}\left(I^N\le x\right)-\int_{-\infty}^x\int_{\mathbb{R}}(2\pi)^{-1}e^{-iuv}\left(e^{-u^2/2}+\psi_1(u)M^{-1/2}+\psi_2(u)NM^{-1}\right)dudv.
\end{align}
Note that $\widetilde{\Psi}_{I^N}(u)$ also satisfies \eqref{eq:ExControl-PsiN}, which admits the representation
\begin{align*}
\int_{\mathbb{R}}e^{iux}d\mathbf{R}_N(x)=&\int_{\mathbb{R}}e^{iux}d\bP^{N,T}_{\theta_0}\left(I^N\le x\right)-\left(e^{-u^2/2}+\psi_1(u)M^{-1/2}+\psi_2(u)NM^{-1}\right)
\\
=&\widetilde{\Psi}_{I^N}(u)-\left(e^{-u^2/2}+\psi_1(u)M^{-1/2}+\psi_2(u)NM^{-1}\right)=\mathbf{r}_N(u).
\end{align*}
Then, following \citet[Lemma 4 p.77 and Theorem 12 p.33]{Cramer1970}), we  obtain that
\begin{align*}
\mathbf{R}_N(x)=O\left(\int_{CN^\nu}^{\infty}\frac{|\widetilde{\Psi}_{I^N}(u)|}{u}du+M^{-1}+NM^{-3/2}+N^2M^{-2}\right),
\end{align*}
for some constant $C>0$, and where the remainder is uniform in $x$.
This, combined with \eqref{eq:ExControl-PsiN}, yields that
\begin{align*}
\mathbf{R}_N(x)=&O\left(O\left(\left(1+\frac{\widetilde{\mu}_1 (CN^\nu)^2}{M}\right)^{-\widetilde{\mu}_2TM/4\lambda_1^{2\beta}}\right)+M^{-1}+NM^{-3/2}+N^2M^{-2}\right)
\\
=&O\left(O\left(e^{-\widetilde{\mu}_3 MN^{-2\nu}}\right)+M^{-1}+NM^{-3/2}+N^2M^{-2}\right)=O\left(M^{-1}+NM^{-3/2}+N^2M^{-2}\right),
\end{align*}
where $\widetilde{\mu}_3$ is some positive constant, and the remainder is uniform in $x$. Thus, if we set
\begin{align*}
\mathfrak{R}_N(x)=\left(M^{-1}+NM^{-3/2}+N^2M^{-2}\right)^{-1}\mathbf{R}_N(x)
\end{align*}
and recall \eqref{eq:RemainDefN}, then \eqref{eq:ProbExpan-IN} for $\delta=0$ follows immediately with
\begin{align*}
\Phi_k(x)=&\int_{-\infty}^x\int_{\mathbb{R}}(2\pi)^{-1}e^{-iuv}\psi_k(u)dudv,\qquad k=1,2.
\end{align*}
Since $\mathbf{R}_N(\cdot)$ is uniform in $x$, we have that $\mathfrak{R}_N(\cdot)$ is uniformly bounded for $N\to\infty$.
Taking into account the form of $\psi_k$, and the definition of $\Phi_k$'s, the  smoothness and boundedness of $\Phi_k(\cdot)$ and its derivatives follow immediately.

Now we consider the case $\delta\neq0$. Since $\Phi(x)$ and $\Phi_k(x)$ have bounded derivatives, we have the expansions
\begin{align*}
\Phi(x+\delta M^{-1/2})=&\Phi(x)+\Phi'(x)\delta M^{-1/2}+O(M^{-1}),
\\
\Phi_k(x+\delta M^{-1/2})=&\Phi_k(x)+O(M^{-1/2}),\qquad 0\le k\le n,
\end{align*}
where the remainder is uniformly bounded in $x$. Note that, from the above case $\delta=0$, and with $x:=x+\delta M^{-1/2}$ in \eqref{eq:ProbExpan-IN}, we already have
\begin{align*}
\bP^{N,T}_{\theta_0}\left(I_T\le x+\delta M^{-1/2}\right)=&\Phi(x+\delta M^{-1/2})+\Phi_1(x+\delta M^{-1/2})M^{-1/2}+\Phi_2(x+\delta M^{-1/2})NM^{-1}\notag
\\
&+\mathfrak{R}_N(x+\delta M^{-1/2})\left(M^{-1}+NM^{-3/2}+N^2M^{-2}\right).
\end{align*}
Combining the above two relations implies \eqref{eq:ProbExpan-IN} with
\begin{align*}
\Phi_1^\delta(x)=\delta\Phi'(x)+\Phi_1(x),\qquad \Phi_2^\delta(x)=\Phi_2(x).
\end{align*}
Then, by explicitly computing $\Phi_1(x)$ and $\Phi_2(x)$, we finish the proof.
\end{proof}

\end{appendix}


\def\cprime{$'$}

}
\end{document}